\documentclass[12pt,a4paper]{article}
\usepackage{amsmath,amsthm,amssymb}%,amsfonts,color}
\usepackage{hyperref}

\newcommand{\R}{{\mathbb{R}}}          % \R       = reais

\newcommand{\g}{{\mathfrak{g}}}       %
\newcommand{\sympol}{{e}}       %
\newcommand{\XIS}{{\mathfrak{X}}}

\newcommand{\rr}{\rightarrow}
\newcommand{\lrr}{\longrightarrow}

\newcommand{\calD}{{\cal D}}             %
\newcommand{\calF}{{\cal F}}             %
\newcommand{\calI}{{\cal I}}             %
\newcommand{\calR}{{{\cal R}^\xi}}             %

\newcommand{\na}{{\nabla}}
\newcommand{\tr}[1]{{\mathrm{tr}}\,{#1}}
\newcommand{\sg}{{\mathrm{sg}}\,}

\newcommand{\dx}{{\mathrm{d}}}
\newcommand{\inv}[1]{{#1}^{-1}}

\newcommand{\cinf}[1]{{\mathrm{C}}^\infty_{#1}}
\newcommand{\vol}{{\mathrm{vol}}}
\newcommand{\ric}{{\mathrm{Ric}\,}}

\newcommand{\tsb}{{{\cal S}}}
\newcommand{\tsbu}{{{S_1M}}}

\newtheorem{teo}{Theorem}[section]

\newtheorem{prop}{Proposition}[section]

\newenvironment{Rema}[1][Remark.]{\begin{trivlist}
\item[\hskip \labelsep {\bfseries #1}]}{\end{trivlist}}

\def\cyclic{\mathop{\kern0.9ex{{+}
\kern-2.2ex\raise-.28ex\hbox{\Large\hbox{$\circlearrowright$}}}}\limits}

\title{A fundamental differential system of
\\
 Riemannian geometry}

\author{R. Albuquerque}

\begin{document}

%\begin{color}{DarkGreen}
%\begin{color}{DarkOrchid4}

\maketitle

%\date{\today} %{21 de Maio de 2011}

\markright{\sl\hfill  R. Albuquerque \hfill}

\begin{abstract}

We discover a fundamental exterior differential system of Riemannian geo\-me\-try; indeed, an intrinsic and invariant global system of differential forms of degree $n$ associated to any given oriented Riemannian manifold $M$ of dimension $n+1$. The framework is that of the tangent sphere bundle of $M$. We generalise to a Riemannian setting some results from the theory of hypersurfaces in flat Euclidean space. We give new applications and examples of the associated Euler-Lagrange differential systems.

\end{abstract}

%\tableofcontents

\ 
\vspace*{3mm}\\
{\bf Key Words:} tangent sphere bundle, Riemannian manifold, exterior differential system, hypersurface, Euler-Lagrange system.
\vspace*{2mm}\\
{\bf MSC 2010:} Primary: 53C21, 58A15, 58A32; Secondary: 53C17, 53C25, 53C38, 53C42

\vspace*{14mm}

\section{Introduction}

It is a remarkable feature of differential geometry that so many questions may be addressed through the theory of exterior differential systems. In the present article we concentrate on a well-known contact system, which we expand to a new global invariantly defined exterior differential system of Riemannian geometry.

Given an oriented smooth Riemannian $n+1$-dimensional manifold $M$, in Section \ref{TnedsoS} we start by recalling the metric contact structure defined by the Sasaki metric and the non-vanishing canonical 1-form $\theta$ on the total space of the constant radius $s>0$ tangent sphere bundle $S_sM\lrr M$. Then we turn to our main purpose, which is the study of a system of natural $n$-forms existing always on $S_sM$:
\begin{equation}
\alpha_0,\ \alpha_1,\ \ldots\,,\ \alpha_n .
\end{equation}
Each of these $n$-forms on the contact $2n+1$-manifold $(S_sM,\theta)$ and their $\cinf{}$ linear combinations assume the natural role of Lagrangian forms. So they induce
variational principles of the underlying exterior differential system. The
study of the Lagrangians $\alpha_i$,\ $0\leq i\leq n$, may be pursued through various fields. In order to be succinct, let us assume $s=1$ for the moment. One simple structural equation is then
\begin{equation}
 *\alpha_i=(-1)^{n-i}\theta\wedge \alpha_{n-i} .
\end{equation}

Suppose we have an \textit{adapted} local coframe on $S_1M$, a basis of 1-forms $e^0$, $e^1,\ldots,e^n$, $e^{n+1},\ldots,e^{2n}$, where $\theta=e^0$, the $e^0,e^1,\ldots,e^n$ are horizontal and the remaining are vertical (the definition of adapted coframe is given below, following the usual tangent manifold tensor decomposition and with the $e^i$ corresponding to $e^{i+n}$). Then, in case $n=1$, we have a global coframing of $\theta$ and the two 1-forms $\alpha_0=e^1$ and $\alpha_1=e^2$. The following formulas for a surface, where $c$ denotes the Gaussian curvature of $M$, are well-known as Cartan Structural Equations. They have been deduced in a purely Riemannian fashion in \cite[Chapter 7.2]{SingerThorpe}:
\begin{equation}
\begin{split} 
  \dx\theta=\alpha_1\wedge\alpha_0 ,\qquad\qquad\ \  \\
 \dx\alpha_1=c\,\alpha_0\wedge\theta ,\quad\qquad \dx\alpha_0=\theta\wedge\alpha_1 .
\end{split}
\end{equation}

In this article, in Theorem \ref{derivadasdasnforms}, we deduce general formulas for $\dx\alpha_i$, $i=0,1,\ldots,n$, in all dimensions. These structural equations lead to new applications. In case $n=2$, we have ($e^{jk}=e^j\wedge e^k$)
\begin{equation}\label{alfascomniguala2}
  \alpha_0=e^{12},\qquad\alpha_1=e^{14}+e^{32},\qquad\alpha_2=e^{34} ,
\end{equation}
and then
\begin{equation}
\begin{split} 
 -\frac{1}{2}\dx\theta\wedge\dx\theta=\alpha_0\wedge\alpha_2
   =-\frac{1}{2}\alpha_1\wedge\alpha_1 ,  \qquad\qquad \\
 \dx\alpha_2=\calR\alpha_2 ,\qquad \dx\alpha_1=2\theta\wedge\alpha_2 -r\,\vol ,
\qquad \dx\alpha_0=\theta\wedge\alpha_1 
\end{split} 
\end{equation}
where $\vol=\theta\wedge\alpha_0$ is the pullback of the volume-form of $M$ and $\calR\alpha_2$ and $r$ are a curvature dependent 3-form and function respectively. 

It is quite remarkable that for any given metric of constant sectional curvature $c$
we have the formulas, for all $i=0,\ldots,n$,
\begin{equation}\label{dalphaicomcsc}
\begin{split}
  \dx\alpha_i &=
\theta\wedge\bigl((i+1)\,\alpha_{i+1}-c(n-i+1)\,\alpha_{i-1}\bigr) , \\
 \dx(*\alpha_i)&=0 .
\end{split}
\end{equation}
For $c=0$ we may say these are the differential versions of the quite well-known Hsiung-Minkowski identities.

We also show how the forms relate to certain calibrated geometries and, at least, to one special Riemannian geometry of ${\mathrm{G}}_2$-\textit{twistor} space. The latter consists of a natural ${\mathrm{G}}_2$ structure existing always on $S_1M$ for any given oriented Riemannian 4-manifold $M$, firstly discovered in \cite{AlbSal1} and which brought a new field of interaction with Einstein metrics.

In Theorem \ref{metricaEinsteinealpha2} we find the essential result which is behind the case of ${\mathrm{G}}_2$-twistors. We prove that an oriented Riemannian $n+1$-manifold is Einstein if and only if $\alpha_{n-2}$ is coclosed. The index $i$ in $\alpha_i$ certainly relates to a degree of Riemannian complexity, notwithstanding $\alpha_n,\alpha_{n-1}$ being always coclosed and all the $n$-forms being coclosed if $M$ has constant sectional curvature.

Further applications of the natural Lagrangians go through an analysis of the
Euler-Lagrange equations of the first few, $i=0,1,2$, functionals
$\calF_i(N)=\int_{\hat{N}}\alpha_i$ on the set of submanifolds $N\hookrightarrow M$
with natural lift $\hat{N}$ into $S_1M$. These follow as applied in similar context by known references in metric problems on space-forms, cf. \cite{BGG,Rei}. Indeed the pullback of the $\alpha_i$ coincides with $(-1)^i$ times the $i$th-symmetric polynomial on the principal curvatures of ${N}$, now in a general Riemannian framework. With these methods we are able to give in Theorem \ref{teo_HsiungMinkowski} a new proof of the Hsiung-Minkowski identities in Euclidean space, cf. \cite{Hsiung,Katsurada}. One may also study problems regarding linear Weingarten equations, whose significance is revealed through the flat metric case in \cite{BGG}, but we do not pursue these here.

Throughout the text the reader will notice that we focus on some of the consequences of \eqref{dalphaicomcsc}. As the equations show, the interplay with submanifold
theory of space-forms seems to be a most promising feature of the new 
$n$-forms. For instance, \eqref{dalphaicomcsc} yields partly the variational principle statement of \cite[Theorem B]{Rei}, cf. Theorem \ref{teo_partofTheoremBofRei}.

The question of finding infinitesimal symmetries and conservation laws of Lagrangian
systems is recalled below and applied to our system. The flat case is known to be quite difficult; we give some results for an ambient manifold of any constant sectional curvature $c$ and for a Lagrangian system with certain constant coefficients.

In searching the literature, we conclude that the differential system of the $\theta$, $\alpha_0,\ldots,\alpha_n$ appears also in \cite[p. 32]{BGG}, in a similar form. It is used in contact systems applied to studies of the geometry of hypersurfaces in Euclidean space $\R^{n+1}$. The particular case of the structural equations \eqref{dalphaicomcsc}, with $c=0$, may thus be said to be already known.

The author of this article is very grateful to a generous referee, whose clever reading and comments much benefited the final text. He is also grateful to a Marie Curie Fellowship he received from the ``FP7 Program'' of the European Union after the findings which can be seen here below.

%From what we realise through the literature, the importance of the subject and
%the name appearing in all references which mention explicitly the
%differential system, the author would agree on the term \textit{Griffiths forms}
%to refer to the $n$-forms $\alpha_i$ of the natural exterior
%differential system of a tangent sphere bundle of an oriented Riemannian
%$n+1$-manifold.

% The latter is henceforth called the \textit{Griffiths exterior differential system}.

%The contents of this article are as follows. In section 2 we present our techniques
%with the Riemannian geometry of $S_sM$. We have in view the description of the
%Griffiths forms and their first structural equations. In section 3 we present the
%first applications and examples, namely to Einstein metrics. There we also look into
%special Riemannian structures and concentrate on variational problems of the
%geometry of hypersurfaces, a case in which we mostly follow \cite{BGG}. In section 4
%we call attention upon the study of the infinitesimal symmetries and conservation
%laws. Finally in section 5 we complete the more computational proofs from previous
%sections.

\section{The natural exterior differential system on $S_sM$}
\label{TnedsoS}

\subsection{Geometry of the tangent sphere bundle}
\label{Geomofthetangentspherebundle}

Let $M$ be an $n+1$-dimensional smooth Riemannian manifold with metric tensor $g=\langle\ ,\ \rangle$. We need to recall some basic differential geometry concepts for the study of the total space of the tangent bundle $\pi:TM\lrr M$. The theory has been thoroughly developed by various mathematicians in the last six decades. Our  technique, recalled below, was introduced in \cite{Alb5,Alb3,Alb4}.

The total space $TM$ is well-known to be a $2n+2$-dimensional smooth manifold. A canonical atlas arises from any given atlas of $M$ and induces a natural isomorphism $V:=\ker\dx\pi\simeq\pi^*TM$ of vector bundles \textit{over} $TM$. This clearly agrees fibrewise with the tangent bundle to the fibres of $TM$. Supposing just a linear connection $\na$ is given on $M$, then the tangent bundle of $TM$ splits as $TTM=H\oplus V$, where $H$ is a sub-vector bundle (depending on $\na$). Clearly the \textit{horizontal} sub-bundle $H$ is also isomorphic to $\pi^*TM$ through the map $\dx\pi$. We may thus define an endomorphism 
\begin{equation}\label{Bendomorphism}
 B:TTM\lrr TTM
\end{equation}
which transforms $H$ into $V$ in the obvious way and vanishes on the \textit{vertical} sub-bundle. $B$ is called the \textit{mirror} map.

There is a canonical vertical vector field $\xi$ over $TM$, well defined by $\xi_u=u,\ \forall u\in TM$. Note that $\xi$ is independent of the connection and the vector $\xi_u$ lies in the vertical side $V_u,\ \forall u\in TM$. Henceforth there exists a unique horizontal canonical vector field, here formally denoted $B^{\mathrm{t}}\xi\in H$, such that $B(B^{\mathrm{t}}\xi)=\xi$. Such vector field is known as the \textit{geodesic spray} of the connection, cf. \cite{Sakai}. Proceeding in this reasoning, we let $\na^*=\pi^*\na$ denote the pullback connection on $\pi^*TM$ and let $(\cdot)^h,\:(\cdot)^v$ denote the projections of tangent vectors onto their $H$ and $V$ components. Then, $\forall w\in TTM$,
\begin{equation}\label{pullbackconnectionproperties}
\na^*_w\xi=w^v\qquad\ \mbox{and}\qquad\ H=\ker(\na^*_\cdot\xi) .
\end{equation}

The manifold $TM$ also inherits a linear connection, still denoted $\na^*$, which is just $\na^*\oplus\na^*$ according with the canonical decomposition
\begin{equation}
 TTM=H\oplus V\simeq\pi^*TM\oplus\pi^*TM .
\end{equation}
Of course, the connecting endomorphism $B$ is parallel for such $\na^*$. Furthermore, the theory tells us that, for a torsion-free connection $\na$, the torsion of $\na^*$ is given, $\forall v,w\in TTM$, by
\begin{equation}\label{torsaodenablaasterisco}
 T^{\na^*}(v,w)=R^{\na^*}(v,w)\xi=\pi^*R^\na(v,w)\xi:=\calR(v,w)
\end{equation}
where $R^{\na^*},R^\na$ denote the curvature tensors (the proof of \eqref{torsaodenablaasterisco} being recalled in Section \ref{Pomf}). This tells us the torsion is vertically valued and only depends on the horizontal directions. In the third identity the tensor $\calR\in\Omega^2_{TM}(V)$ is defined.

We finally turn to the given metric tensor $g$ on $M$. Recall the Sasaki metric
on $TM$, here also denoted by $g$, is given naturally by the pullback of the metric
on $M$ both to $H$ and $V$. The restriction $B:H\rr V$ of the mirror map is then a parallel and metric-preserving morphism. We employ the notation $B^{\mathrm{t}}$ for the adjoint endomorphism of $B$, which equals the inverse in the case of the referred restriction. Moreover, we have that $J=B-B^{\mathrm{t}}$ is the well-known almost complex structure on $TM$. In particular our manifold is always oriented. Of course, $\na g=0$ implies $\na^*g=0$.

Certainly one may continue to establish the construction with any metric connection, but now we assume the given $\na$ is the Levi-Civita connection of $M$.

Let us  consider the tangent sphere bundle $S_sM\lrr M$ with any fixed constant
radius $s>0$,
\begin{equation}
 S_sM=\{u\in TM:\ \|u\|=s\} .
\end{equation}
This hypersurface is also given by the locus of $\langle\xi,\xi\rangle=s^2$. Using
\eqref{pullbackconnectionproperties} we deduce $T(S_sM)=\xi^\perp\subset
TTM$. Since $TM$ is orientable, $S_sM$ is also always orientable --- the restriction of $\xi/\|\xi\|$ being a unit normal. Moreover, for any $u\in TM\backslash0$, we may find a local horizontal orthonormal frame $e_0,e_1\ldots,e_n$, on a neighbourhood of $u\neq0$, such that $e_0=B^{\mathrm{t}}\xi/\|\xi\|$ (the existence of such moving frame relies on the smoothness of the Gram-Schmidt process and the action of the orthogonal group).

Note that any frame in $H$ extended with its mirror in $V$ clearly determines an orientation on the manifold $TM$. We always adopt the order `first $H$ then $V$'.

The mere existence of a parallel $n+1$-form on $M$ would be interesting enough for our purposes, but here we assume once and for all the manifold is oriented. We let
\begin{equation}\label{alpha_n_concept}
\alpha_n\, =\, \frac{\xi}{\|\xi\|}\lrcorner(\inv{\pi}\vol_M)
\end{equation}
where $\inv{\pi}\vol_M$ is the vertical pullback of the volume-form of $M$.
And let $\vol$ denote the pullback by $\pi$ of the volume-form of $M$:
\begin{equation}
 \vol=\pi^*\vol_M .
\end{equation}
With the dual horizontal coframing $\{e^0,e^1,\ldots,e^n\}$, where $e^0=e_0^\flat$, clearly the identity $\vol=e^0\wedge\cdots\wedge e^n$ is satisfied. Joining in the \textit{mirror} subset $\{\frac{\xi^\flat}{\|\xi\|},e^{n+1},\ldots,e^{2n}\}$, this is, the forms defined by $e^{n+i}(e_{j+n})=e^i(e_j)=\delta_j^i$, 
$e^{n+i}(e_j)=e^i(e_{j+n})=0$, $\forall i,j$,
we may then use the following volume-form of $TM$:
\begin{equation}
  \begin{split}
 \vol_{TM}\ =&\ (-1)^{n+1} e^0\wedge e^1\wedge\cdots\wedge e^n\wedge\frac{\xi^\flat}{\|\xi\|}\wedge
e^{n+1}\wedge\cdots\wedge e^{(2n)} \\=\ &
\frac{1}{s}\,\xi^\flat\wedge\vol\wedge\alpha_n .
  \end{split}
\end{equation}
In this way, the orientation $(\xi/s)\lrcorner\vol_{TM}$ of the Riemannian submanifold $S_sM$ agrees with
$\vol\wedge\alpha_n=e^{01\cdots(2n)}=e^0\wedge e^1\wedge\cdots\wedge e^n\wedge\cdots\wedge e^{(2n)}$. 
We shall assume always the canonical orientation is given by $\vol\wedge\alpha_n$ on $S_sM$ (we omit the notation of the restriction of those forms to the hypersurface). A direct orthonormal frame as the one introduced previously is said to be \textit{adapted}.

To ease the notation we let $\tsb$ denote the total space $S_sM$ with any $s$ freely chosen, only recalled when necessary.

The submanifold $\tsb$ admits a metric linear connection $\na^\divideontimes$, as we shall see next. For any vector fields $y,z$ on $\tsb$, the covariant derivative $\na^*_yz$ is well-defined and, admitting $y,z$ perpendicular to $\xi$, we just have to add a correction term:
\begin{equation}\label{nablastar}
  \begin{split}
 \na^\divideontimes_yz\ =\ \na^*_yz-\frac{1}{s^2}\langle\na^*_yz,\xi\rangle\xi \ =\ 
\na^*_yz+\frac{1}{s^2}\langle y^v,z^v\rangle\xi .
  \end{split}
\end{equation}
Since $\langle\calR(y,z),\xi\rangle=0$, we see from (\ref{torsaodenablaasterisco}) that a torsion-free connection $D$ is given by $D_yz=\na^\divideontimes_yz-\frac{1}{2}\calR(y,z)$. We remark $D$ is not the Levi-Civita connection in general, with further details on homotheties, topology and metric connections on $\tsb$ being found in \cite{Alb5,Alb3,Alb4}.

\subsection{The contact structure and the fundamental differential forms}
\label{Thenewnforms}

Continuing the above, we let $\theta$ denote the 1-form on $\tsb$
\begin{equation}\label{mu}
\theta=\langle\xi,B\,\cdot\,\rangle=s\,e^0 .
\end{equation}
The differential geometry of $\tsb$ is much determined by the following proposition, whose proof shall be recalled in Section \ref{Pomf}. The result was essentially deduced by Y.~Tashiro through chart computations, cf. \cite{Blair,Tash}.
\begin{prop}[Y.~Tashiro]      \label{dmu}
 We have $\dx\theta=e^{(1+n)1}+\cdots+e^{(2n)n}$. In other words, $\forall v,w\in
T\tsb$,
\begin{equation}\label{dmu1}
 \dx\theta(v,w)=\langle v,Bw\rangle-\langle w,Bv\rangle .
\end{equation}
\end{prop}
It follows that $(\tsb,\theta)$ is a contact manifold. In fact, $
 \theta\wedge(\dx\theta)^n=(-1)^{\frac{n(n+1)}{2}}n!s\,\vol\wedge\alpha_n\neq0$,
as we shall care to establish later.

We observe here that the expression of $\dx\theta$ is not linear in $s$, lest one
should be driven to conclude so. This shows the relative importance of the computations being done with any radius albeit constant.
\begin{Rema}
We may also describe a \textit{metric} contact structure on $\tsb$. Finding the correct
weights on the fixed metric, the 1-form $\theta$ and the so-called Reeb vector field,
which of course is a multiple of $B^{\mathrm{t}}\xi$, gives
\begin{equation}
 \hat{g}=\frac{1}{4s^2}g,\quad \hat{\xi}=2B^{\mathrm{t}}\xi,\quad
\eta=\hat{g}(\hat{\xi},\,\cdot\,)
=\frac{1}{2s^2}\theta,\quad \varphi=B-B^{\mathrm{t}}-2\xi\otimes\eta .
\end{equation}
Then $\eta(\hat{\xi})=1,\ \varphi(\hat{\xi})=0,\
\varphi^2=-1+\eta\otimes\hat{\xi},\
\hat{g}(\varphi\,\cdot\,,\varphi\,\cdot\,)=\hat{g}-\eta\otimes\eta$ and
$\dx\eta=2\hat{g}(\,\cdot\,,\varphi\,\cdot\,)$ as required. Such metric contact structure is Sasakian if and only if $M$ has constant sectional curvature $1/s^2$, a result which is first proved for $s=1$ by Tashiro in \cite{Tash}.
\end{Rema}

Before finally introducing the $n+1$ natural $n$-forms on $\tsb$ with any constant radius $s$, let us define for any $0\leq i\leq n$ the rational number
\begin{equation}\label{enei}
  n_i=\frac{1}{i!(n-i)!} .
\end{equation}
Continuing with the notation and the adapted frame introduced earlier, we have first
with $\inv{\pi}\vol_M$ the vertical pullback, cf. \eqref{alpha_n_concept},
\begin{equation}\label{alpha}
\alpha_n\, =\, \frac{\xi}{\|\xi\|}\lrcorner(\inv{\pi}\vol_M)=
e^{(n+1)}\wedge\cdots\wedge e^{(2n)}.
\end{equation}
Then, finally, for each $i$, we define the $n$-form $\alpha_i$ by
\begin{equation}\label{alpha_i}
 \alpha_i=n_i\,\alpha_n\circ(B^{n-i}\wedge1_{T\tsb}^{i}) ,
\end{equation}
this is, $\forall v_1,\ldots,v_n\in T\tsb$,
\begin{equation}\label{alpha_itravez}
 \alpha_i(v_1,\ldots,v_n) = n_i\sum_{\sigma\in S_n} 
  \mathrm{sg}(\sigma)\,\alpha_n(Bv_{\sigma_1},\ldots,Bv_{\sigma_{n-i}}, v_{\sigma_{n-i+1}},\ldots,v_{\sigma_n}) .
\end{equation}
Of course, $S_n$ denotes the symmetric group and
$1=1_{T\tsb}$ denotes the identity en\-do\-mor\-phism of $T\tsb$. Note that $B^{i}=\wedge^{i}B=B\wedge\cdots\wedge B$ with $i$ factors. The notation
$\alpha\circ(B^{i}\wedge1^{n-i})$ shall be duly justified in Section \ref{Pomf}. Notice $\alpha_n$ is unambiguously defined since $\alpha_n\circ\wedge^n1=n!\,\alpha_n$. We remark also that $\alpha_0=e^{1\ldots n}$, which justifies the introduction of the weight $n_i$. For convenience of notation we define $\alpha_{n+1}=\alpha_{-1}=0$.

The reader may see all the $\alpha_i$ for the cases $n=1,2$ in the Introduction
section.

Let $*$ denote the Hodge star-operator on $\tsb$. Of course, it satisfies
$**=1_{\Lambda^*_{\tsb}}$. 
\begin{prop}[Basic structure equations]   \label{Basicstrutequations}
For any $0\leq i\leq n$ we have:
\begin{equation}   \label{Basicstrutequations1}
 *\theta=
s\,\alpha_0\wedge\alpha_n=\frac{s(-1)^{\frac{n(n+1)}{2}}}{n!}\,(\dx\theta)^{n}  ,
\end{equation}
\begin{equation}\label{Basicstrutequations2}
*\,(\dx\theta)^i=
(-1)^{\frac{n(n+1)}{2}}\frac{i!}{(n-i)!s}\,\theta\wedge(\dx\theta)^{n-i}  ,
\qquad\quad
*\alpha_i=\frac{(-1)^{n-i}}{s}\,\theta\wedge\alpha_{n-i}  .
\end{equation}
Moreover,
\begin{equation}
 \alpha_i\wedge\dx\theta=0 \qquad \mbox{and} \qquad \alpha_i\wedge\alpha_j=0,\ \ \forall j\neq
n-i.
\end{equation}
\end{prop}
The proof of all propositions and theorems in this sub-section is postponed to Section \ref{Pomf}. Next we use the notation $R_{lkij} = \langle R^\na(e_i,e_j)e_k,e_l\rangle$ with 
\begin{equation}
 R^\na(e_i,e_j)e_k=\na_{e_i}\na_{e_j}e_k-\na_{e_j}\na_{e_i}e_k-\na_{[e_i,e_j]}e_k .
\end{equation}
\begin{teo}[1st-order structure equations] \label{derivadasdasnforms}
We have
\begin{equation}\label{dalphai}
 \dx\alpha_i=\frac{1}{s^2}(i+1)\,\theta\wedge\alpha_{i+1}+\calR\alpha_i
\end{equation}
where
\begin{equation}\label{Ralphai}
 \calR\alpha_i=\sum_{0\leq j<q\leq n}\sum_{p=1}^nsR_{p0jq}\,e^{jq}\wedge
e_{p+n}\lrcorner\alpha_i .
\end{equation}
\end{teo}
We have in particular the formulas $\calR\alpha_0=0,\ \calR\alpha_{1}=-rs\,\vol=-r\,\theta\wedge\alpha_0$, where $r=\frac{1}{s^2}\inv{\pi}\ric(\xi,\xi)$ is a smooth function on $\tsb$ determined by the Ricci curvature of $M$.  In other words, $r$ is defined by
\begin{eqnarray}\label{defofr}
  u\in\tsb\:\longmapsto\: r(u)&=&  \frac{1}{s^2}\ric_{\pi(u)}(u,u)  \nonumber \\
     &=&\frac{1}{s^2}\tr{R^\na_{\pi(u)}(\,\cdot\,,u)u}=\sum_{j=1}^nR_{j0j0} . 
\end{eqnarray}
We therefore write
\begin{equation}\label{dalphazero}
 \dx\alpha_0=\frac{1}{s^2}\,\theta\wedge\alpha_{1}
\end{equation}
and
\begin{equation}\label{dalphaum}
 \dx\alpha_{1}=\frac{2}{s^2}\,\theta\wedge\alpha_{2}-r\,\theta\wedge\alpha_0 .
\end{equation}

Also we remark
\begin{eqnarray}\label{Ralphan}
 \dx\alpha_n&=&\calR\alpha_n   \nonumber \\
&=&\sum_{0\leq j<q\leq n}\sum_{p=1}^n(-1)^{p-1}sR_{p0jq}\,e^{jq}
\wedge e^{(n+1)\cdots\widehat{(n+p)}\cdots (2n)}   .
\end{eqnarray}
Since
\begin{equation}
 \dx(\calR\alpha_i)=\frac{1}{s^2}(i+1)\,\theta\wedge\calR\alpha_{i+1},
\end{equation}
we have
\begin{equation}
 \dx\theta\wedge\calR\alpha_i=0.
\end{equation}

We recall the decomposition of $n$-forms under the trivial plus twice the standard representation of the special orthogonal group ${\mathrm{SO}}(n)$ on $\R^{2n+1}=\R\oplus2\R^n$. This induces the decomposition $\Lambda^n(\R^{2n+1})= \oplus_{l+p+q=n}\Lambda^l(\R)\otimes\Lambda^p(\R^n)\otimes\Lambda^q(\R^n)$.
It has as invariants, giving the 1-dimensional representations, the forms $\theta^{\varepsilon}\wedge(\dx\theta)^{[\frac{n}{2}]}$, $\alpha_0,\ldots,\alpha_n$, where $\varepsilon=0$ or $1$ according to the parity of $n$.

We call the $\dx$-closed ideal generated by the set of differential forms
$\theta,\alpha_1,\ldots,\alpha_n$ the \textit{natural exterior differential
system on the tangent sphere bundle of $M$}. 

On the case of flat space $M=\R^{n+1}$ our system coincides essentially with that defined by Bryant, Griffiths and Grossman in \cite[p. 32]{BGG} for the purpose of the study of Euclidean submanifolds.

Case $n=1$ is due to Cartan, cf. \cite{SingerThorpe}, where $S_sM\lrr M$ is seen as the principal $\mathrm{SO}(2)$-bundle of norm $s$ orthogonal frames. The generalisation is thus overall in the degree of the forms, which grows linearly with $n$. The dimension of $S_sM$ grows linearly too, whereas that of any Cartan's principal frame bundle, with the celebrated Lie algebra valued connection 1-forms, has quadratic growth with $n$.

To the best of our knowledge, there exist in the literature only a few descriptions of the exterior differential system of $\theta$ and the $\alpha_i$. In each, we find a distant or partial relation with our differential system, on some fibre bundle over a given manifold. The closest system, in \cite{BGG}, has in view the solution of some geometrical problems in Euclidean space. We also find this kind of structural $n$-forms in \cite{Via}, defined over Cartan's principal conformal frame bundle, with a large symmetry group. However, the structural equations turn out quite apart from the present, even for the case they were designed, i.e., a conformally flat base manifold.

With the remarkable applications of differential systems within mechanical and geometrical motivations seen in \cite{Gri1}, we are presented with local solutions which arguably recur to the 1- and 2-forms, respectively, of 2 and 3 dimensional base $M$. In the landmark reference \cite[p.152]{BCGGG} the emphasis remains in 3-dimensional questions, and the same is true for \cite{Gri2} or \cite{IveyLan}.

\subsection{Some natural geometric properties}
\label{subsection_Somefunctorialproperties}

We continue with $M$ of dimension $n+1$ and $S_sM$ with any radius $s>0$. In this section we leave some computational proofs to Section \ref{Pomf}.

It is trivial to see  $\delta\theta=0$. As well as the following result.
\begin{prop}\label{alphanealphanmenos1arecoclosed}
The differential forms $\alpha_n$ and $\alpha_{n-1}$ are always coclosed.
\end{prop}
\begin{proof}
Indeed, we have that
$\dx*\alpha_{n-1}=\frac{-1}{s}\,\dx(\theta\wedge\alpha_{1})=
\frac{-1}{s}\,\dx\theta\wedge\alpha_{1}+\frac{1}{s}\,
\theta\wedge\dx\alpha_{1}=0$ and
$\dx*\alpha_n=\dx\,\vol=\pi^*\dx\,\vol_M=0$.
\end{proof}
Uniqueness of the Levi-Civita connection implies that any isometry preserves the horizontal distributions. Hence the following trivial result.
\begin{prop}
 Let $f:M\lrr N$ be an isometry between oriented Riemannian manifolds. Then the map $\hat{f}:S_sM\lrr S_sN$ induced by the differential of $f$ is an isometry and a contactomorphism which applies by pullback each $\alpha_i$ of $N$ to the respective $\alpha_i$ of $M$.
\end{prop}
Let us see how the differential system behaves under restriction to totally geodesic hypersurfaces.
\begin{teo}\label{teorema_restrictiontoTGeodhypers}
Let $f:N\lrr M$ be an $n$-dimensional orientable embedded totally geodesic Riemannian hypersurface. Let $\hat{f}$ denote the induced embedding $S_sN\lrr S_sM$ and let $\vec{n}$ denote the unit normal vector-field of $N$ in $M$, inducing through $\vec{n}\lrcorner\vol_M$ the orientation of $N$. Then the differential system on $S_sN$ arises in both ways through the following formulas, for every $i=0,\ldots,n-1$:
\begin{equation}\label{restrictiontoTGeodhypers_formula}
 \alpha^N_i=-{\hat{f}}^*(\vec{n}^v\lrcorner\alpha_{i+1})=
-{\hat{f}}^*(\vec{n}^h\lrcorner\alpha_{i})  .
\end{equation}
%Moreover, $\dx\alpha^N_i=-(\mathrm{n}^h+\mathrm{n}^v)\lrcorner\dx\alpha_i$, $\forall i=0,\ldots,n-1$.
\end{teo}
\begin{proof}
 Since the Levi-Civita connection of $N$ is the same as that of $M$, the horizontal distribution of $N$ is mapped by $\hat{f}_*$ into that of $M$. Hence the mirror map on $TTN$ is obtained by simply restricting $B$ on both its source and target spaces. Clearly $\xi$ restricts always to the tautological vector field of $N$. Since we know the induced volume-form on $N$ it follows that $\alpha^N_{n-1}=-{\hat{f}}^*(\vec{n}^v\lrcorner\alpha_n)$. Since $B\vec{n}^v=0$ and $\vec{n}^h\lrcorner\alpha_n=0$, the proof is just a few computations away, though long, which we postpone to the section entirely concerned with them.
\end{proof}

Another differential system worth studying is the one induced on $\partial S_sM$, which contains $S_s\partial M$, for manifolds with boundary.

Now we define a new 1-form $\rho$ to be $\frac{1}{s}\xi\lrcorner\inv{\pi}\ric$ restricted to $\tsb$. We are referring to the vertical lift of the Ricci tensor, then pulled-back to $\tsb$. In other words, using an adapted frame,
\begin{equation}\label{defofrho}
 \rho=\sum_{a,b=1}^nR_{a0ab}\,e^{b+n} .
\end{equation}
Recall $\alpha_n,\alpha_{n-1}$ are always coclosed (Proposition \ref{alphanealphanmenos1arecoclosed}). The next theorem gives one step further; it
transforms the equation for an Einstein metric into a formally degree 1 equation.
\begin{teo}\label{metricaEinsteinealpha2}
We have
\begin{equation}\label{dasteriscoalphadois}
\dx*\alpha_{n-2}\:=\:s\,\rho\wedge\vol .
\end{equation}
The metric on $M$ is Einstein if and only if $\delta\alpha_{n-2}=0$.
\end{teo}
\begin{proof}
We leave the deduction of \eqref{dasteriscoalphadois} to the computations section, below. The conclusion on the metric being Einstein is found by observing how
$\rho$ is defined. Indeed, $\delta\alpha_{n-2}=0$ if and only if at each point $u\in\tsb$ the form $\xi\lrcorner\inv{\pi}\ric=\lambda\xi^\flat$ for some real function $\lambda$ of $u$. This says $\ric(u,u)=\lambda\langle u,u\rangle$, which is the Einstein metric condition.
\end{proof}

\subsection{Examples and applications to special metrics}
\label{Atte}

Let us see some examples of the fundamental equations associated to the tangent sphere bundle of a given oriented Riemannian manifold $M$ of dimension $n+1$.

\vspace{3mm}

\noindent\textsc{Example 1}. From the definitions, in case $n=1$ we have a global coframe
\begin{equation}
  \theta=s\,e^0\ ,\qquad\ \alpha_0=e^1\ ,\qquad\ \alpha_1=e^2 .
\end{equation}
Now $\tsb$ coincides with the total space of a principal $\mathrm{SO}(2)$-bundle and is hence a parallelisable manifold for every surface --- as it is well-known. The exterior differential system agrees with the Cartan structural equations, cf. \cite{Alb2015a,SingerThorpe}:
\begin{equation}\label{dalphais_m=2}
  \dx\theta=\alpha_1\wedge\alpha_0\ ,\qquad \dx\alpha_1=c\,\alpha_0\wedge\theta\ ,
\qquad \dx\alpha_0=\frac{1}{s^2}\,\theta\wedge\alpha_1 
\end{equation}
where $c=R_{1010}$ is the Gaussian curvature of $M$. It is quite interesting to recognise that $c$ is non-constant in general, although a fibre constant.

Let us see the above equations in the trivial case of $S_s\R^2=\R^2\times S^1_s$.
Admitting coordinate functions $(x^1,x^2,u_1,u_2)$, subject to $(u_1)^2+(u_2)^2=s^2$,
we immediately find $s\,e^0=u_1\dx x^1+u_2\dx x^2,\,\ s\,e^1=-u_2\dx x^1+u_1\dx
x^2,\,\ s\,e^2=-u_2\dx u_1+u_1\dx u_2$. Using the identity $u_1\dx u_1+u_2\dx u_2=0$, we deduce straightforwardly the equations in \eqref{dalphais_m=2}. Notice a trivial proof for non-flat coordinates is not quite at hand.

\vspace{3mm}

\noindent\textsc{Example 2}. Suppose $n=2$ and $s=1$. Then
\begin{equation}\label{alfascomniguala2-1}
  \alpha_0=e^{12},\qquad\alpha_1=e^{14}+e^{32},\qquad\alpha_2=e^{34} .
\end{equation}
The equations for $\alpha_i,\ i=0,1,2$, give us the global tensors $\calR\alpha_0=0,\ \,\calR\alpha_1=-r\,\vol$,
\,cf. \eqref{defofr}, and the very interesting new global tensor $\calR\alpha_2=$
\begin{equation}\label{Ralpha0}
R_{1001}e^{014}+R_{1002}e^{024}+R_{1012}e^{124}-
 R_{2001}e^{013}-R_{2002}e^{023}-R_{2012}e^{123}  .
\end{equation}
Henceforth
\begin{equation}
\begin{split}
 -\frac{1}{2}\dx\theta\wedge\dx\theta=\alpha_0\wedge\alpha_2=-\frac{1}{2}\alpha_1\wedge\alpha_1 , \qquad \qquad\\
  \dx\alpha_2=\calR\alpha_2\ ,\qquad \dx\alpha_1=2\theta\wedge\alpha_2 -r\,\vol ,
\qquad \dx\alpha_0=\theta\wedge\alpha_1 .
\end{split}
\end{equation}
Presently illustrative, these tensors play an important role in the case of 3-ma\-ni\-folds. We initiate the study in \cite{Alb2015a}. One may also discuss the appearance in \cite[p. 461]{Gri2} of some relation with these equations, locally, arising from the principal frame bundle of $M$ and having in view an example of a hyperbolic differential system.

\vspace{3mm}

\noindent\textsc{Example 3}. It is quite interesting to consider the case of constant sectional curvature $c$ in any dimension $n+1$. Since the Riemann curvature tensor is
$R_{lkqp}=c(\delta_{lq}\delta_{kp}-\delta_{lp}\delta_{kq})$, we prove in Section
\ref{Pomf} that $\calR\alpha_i=-c(n-i+1)\,\theta\wedge\alpha_{i-1}$. Recalling $\alpha_{-1}=\alpha_{n+1}=0$, we thus have
\begin{equation}\label{dalphaicsc}
  \dx\alpha_i=\theta\wedge\bigl(\frac{1}{s^2}(i+1)\,\alpha_{i+1}-c(n-i+1)\,\alpha_{i-1}\bigr) . 
\end{equation}
Notice $\calR\alpha_{1}=-snc\,\vol$, just as expected through \eqref{dalphaum}. Despite
the awkward context, we may formally compare the deduced formula with the
Frenet equations of a curve in $\R^n$ described in \cite[p. 23]{Gri2}. Furthermore, if $M$ has constant sectional curvature, it is easy to see that  $\dx*\alpha_i=0, \,\forall 0\leq i\leq n$. 

\vspace{3mm}

With little effort we deduce the converse of the last implication. Also, since \eqref{dalphaicsc} is obtained independently of $c$ varying or not, the differential system yields a new proof of the Theorem of Schur, by purely skew-symmetric tensor methods.
\begin{prop}\label{dasteriscoalphaicsc}
Suppose $\dim M\geq 3$ and $M$ is connected.\\
(i) The metric is of fibrewise constant sectional curvature if and only if $\dx\alpha_n=-c\,\theta\wedge\alpha_{n-1}$. Also, if and only if $\delta\alpha_i=0,\ \forall i$.\\
(ii) (\emph{Schur Theorem}) Any of the hypotheses in (i) imply $M$ has constant sectional curvature.
\end{prop}
\begin{proof}
 (i) From \eqref{Ralphan} we see that $\calR\alpha_n=-c\,\theta\wedge\alpha_{n-1}$ implies $R_{p00q}=0$, for all $p\neq q>0$. This implies $R_{lkqp}$ of the \textit{constant} type above. Then, using the basic structure equations \eqref{Basicstrutequations2},  equation \eqref{dalphaicsc} easily implies the whole differential system is coclosed. So we are left to prove that $\delta\alpha_i=0$ implies $R_{lkqp}$ of the above fibrewise constant form. Indeed, we immediately find $\dx*\alpha_i=\frac{(-1)^{n-i}}{s}\theta\wedge\calR\alpha_{n-i}$, which vanishes if and only if $R_{p0jq}=0$ for all $p,j,q>0$.\\
 (ii) Let us suppose the algebraic curvature tensor is of the constant type, thus with fibrewise constant $c$. Then $c$ is the pullback to $\tsb$ of a function on $M$. Differentiating \eqref{dalphaicsc} again, we find in particular $\dx c\wedge\theta\wedge\alpha_{n-1}=0$. Writing $\dx c=\sum_{i=0}^n\dx c(e_i)e^i$, with an adapted frame, it then follows that $\dx c(e_i)=0,\ \forall i=1,\ldots,n$. Then $\dx c=a\,\theta$ is some multiple of $\theta$. But $0=\dx a\wedge\theta+a\,\dx\theta$ shows $a=0$.
\end{proof}

Now let us see an application which was found before the present construction of
the natural exterior differential system of a Riemannian manifold. We shall need to refer to concepts of ${\mathrm{G}}_2$ geometry which the reader may follow e.g. in \cite{Joy}.

In \cite{AlbSal1} it is proved that the total space of the unit
tangent sphere bundle $S_1M\rr M$ of any given oriented Riemannian 4-manifold $M$
carries a natural ${\mathrm{G}}_2$-structure. The structure led us to introduce the
3-forms $\alpha_0,\alpha_1,\alpha_2,\alpha_3$,
\begin{equation}%\label{alfascomniguala3}
\alpha_0=e^{123},\quad
\alpha_1=e^{126}+e^{234}+e^{315},\quad\alpha_2=e^{156}+e^{264}+e^{345}, \quad\alpha_3=e^{456} ,
\end{equation}
and to deduce the respective structural differential equations. The space is now called \textit{${\mathrm{G}}_2$-twistor} or \textit{gwistor space}, due to similar techniques of twistor space, and its fundamental structure 3-form is proved to be the case $t=0$ for the following variation of $\mathrm{G}_2$-structures on $S_1M$ keeping the Sasaki metric, cf. \cite{Alb2013b}:
\begin{equation}\label{gwistor}
 \phi_t=\theta\wedge\dx\theta\pm\sin t\,(\alpha_0-\alpha_2)+\cos t\,(\alpha_1-\alpha_3) .
\end{equation}
This incursion within exceptional Riemannian geometry is rather fortunate with
the Hodge dual of $\phi$. The structure 3-form $\phi_0$ is coclosed (a relevant
${\mathrm{G}}_2$ condition since it originates a calibration) if and only if the 4-manifold $M$ is Einstein.

Another important question is then to find the conditions for a linear combination
$\varphi=\sum_{i=0}^n b_i\alpha_i+c\theta^{\varepsilon}\wedge
(\dx\theta)^{[\frac{n}{2}]}$, with $b_i,c\in\cinf{{S_1M}}$, to
be a calibration of degree $n$. Recall a calibration is a closed
$p$-form $\varphi$ such that $\varphi_{|V}\leq\vol_V$ for every oriented tangent
$p$-plane $V$. 

For even $n$ we certainly have an obvious $\varphi$. For $n=1$ the question may be solved easily recurring to \eqref{dalphais_m=2}. For $n=2$ and $3$ we have a complete classification in \cite[Theorems 4.3.2 and 4.3.4]{Joy} of all the possible calibrations which may occur point-wise, in the algebraic form, as elements of $\Lambda^n(\R^{2n+1})$. Several cases do follow as the $\varphi$ we have referred. For $n=3$, we have the ${\mathrm{G}}_2$ structure of gwistor space $\phi$ and $*\phi$, and their variations. The problem, still open, aims towards the study of the associative and coassociative submanifolds of $S_1M$.

\section{The Euler-Lagrange differential systems}
\label{eaa1}

\subsection{Recalling the theory}
\label{TELs}

We wish to study the Euler-Lagrange system $(\tsb,\theta,\dx\theta,\alpha_l)$ for some fixed $\alpha_l$, the Lagrangian form,
in the context of Riemannian geometry. The theory is that of exterior differential
systems of contact manifolds --- explained very clearly in \cite{BGG} as well as in various parts of \cite{Gri2}; we recall here some of its fundamentals.

In this section we assume $(S,\theta)$ is any given contact manifold, not necessarily metric, of dimension $2n+1$.

The \textit{contact differential ideal} $\calI$ is defined as the $\dx$-closed ideal
generated by $\theta\in\Omega^1_S$. In other words, $\calI$ is the ideal
algebraically generated by exterior multiples of $\theta$ and $\dx\theta$. A
\textit{Legendre submanifold} of $S$ consists of an $n$-dimensional manifold $N$
together with an
immersion $f:N\rr S$ such that $f^*\theta=0$. The same is to say $N$ is an 
integral $n$-submanifold, the expression \textit{integral} meaning 
$f^*\calI=0$. We also recall that there exists a generalisation of the
famous Darboux Theorem, which concludes that
certain generic Legendre submanifolds appear as the zero locus of $n+1$
of the so-called Pfaff coordinates. A Legendre submanifold is said to be $\mathrm{C}^1$-\textit{differentiable close} to such a generic Legendre submanifold $N$ if it appears as the graph of a
function on $N$ in the remaining Pfaff coordinates. These are then called the
\textit{transverse} Legendre
submanifolds. The local model is the 1-jet manifold $S=J^1(\R^n)$ of 
Euclidean flat space, with coordinates $(z,x^i,p_i)$, the contact form $\theta=\dx
z-\sum^n_{i=1}p_i\,\dx x^i$ and submanifold $N$ given by $z=0,p_i=0$. Any other $N$
close to that one being of the form $\{(f(x),x^i,\partial_if)\}$. Equivalently, a
Legendre manifold $N$ is considered $\mathrm{C}^1$-close or transverse if $(\dx
x^1\wedge\cdots\wedge\dx x^n)_{|N}\neq0$.

A \textit{Lagrangian} defined over $S$ is just a $\Lambda\in\Omega^n$. It gives
rise to a functional on the set of smooth, compact Legendre submanifolds $N\subset
S$, possibly with boundary, defined by (the restriction or
pullback to $N$ of the Lagrangian is omitted, as the emphasis is on the global and independent $\Lambda$):
\begin{equation}\label{Lagrangianfunctional}
 \calF_\Lambda(N)=\int_N\Lambda .
\end{equation}
There are two types of equivalence in a unique relation associated to such
$n$-forms. An equivalence class $[\Lambda]$ is represented by any element of
$\Lambda+\calI^n+\dx\Omega^{n-1}$, where $\calI^n=\calI\cap\Omega^n$. Such Lagrangian
class clearly induces the same functional on Legendre submanifolds without boundary.
On the other hand, a point-wise algebraic identity carries over to the whole contact manifold, giving:
\begin{equation}\label{igualdade1}
 \calI^k=\Omega^k\,,\quad\forall k>n .
\end{equation}
Hence we have that $\dx\Lambda\in\calI^{n+1}$ and so the above class is well
defined in the co\-ho\-mo\-lo\-gy ring of degree $n$ for the differential complex
$(\Omega^n/\calI^n,\dx)$, i.e. the
\textit{characteristic cohomology} ring $\bar{H}^n(S)$ of the exterior differential
system $(S,\theta)$. It relates to de\,Rham cohomology via the short exact sequence
$0\rr\calI^k\rr\Omega^k\rr\Omega^k/\calI^k\rr0$.

Having chosen a Lagrangian $\Lambda$, we wish to study the functional $\calF_\Lambda$. By \eqref{igualdade1} there exist two forms $\alpha,\beta$ on $S$ such that 
\[   \dx\Lambda=\theta\wedge\alpha+\dx\theta\wedge\beta=
\theta\wedge(\alpha+\dx\beta)+\dx(\theta\wedge\beta )   , \]
By \cite[Theorem 1.1]{BGG} there exists a unique global exact form $\Pi$ such that
$\Pi\wedge\theta=0$ and $\Pi\equiv\dx\Lambda$ in ${H}^{n+1}(\calI)$:
\begin{equation}\label{PoncareCartanform}
\Pi=\dx(\Lambda-\theta\wedge\beta)=\theta\wedge(\alpha+\dx\beta)   .
\end{equation}
The $n+1$-form $\Pi$ is called the \textit{Poincar\'e-Cartan} form.

Now suppose we have a variation of Legendre submanifolds with fixed boundary, i.e.
suppose there exists a smooth map $F:N\times[0,1]\rr S$ such that each $F_{|{N_t}}$,
with $N_t={N\times\{t\}}$, defines a Legendre submanifold (in other words,
$F^*\theta\equiv0\mod\dx t$) and $\partial (F(N_t))$ is independent of $t$. The variation
is of course around the Legendre submanifold $F_{|{N_0}}=f:N\rr S$. Then, subtracting
$\int_{N_t}\theta\wedge\beta=0$ from the left hand side, applying the well-known
formula of the usual derivative becoming the Lie derivative under the integral and
using the Cartan formula, it is proved that 
\begin{equation}
 \frac{\dx}{\dx t}\int_{N_t}\Lambda= 
\int_{N_t}\frac{\partial}{\partial t}\lrcorner\Pi .
\end{equation}
Recurring to usual variational calculus notation, for every varying direction vector field
$v\in\Gamma_0(N;f^*TS)$ vanishing along $\partial N$, at point $t=0$ playing the role
of $\frac{\partial}{\partial t}$, the previous identity reads
\begin{equation}
 \begin{split}
  {\boldsymbol{\delta}}\calF_\Lambda(N)(v) \ &=\ \int_Nv\lrcorner f^*\Pi \\
 &=\ \int_N(v\lrcorner f^*\theta)\,f^*\Psi . 
 \end{split}                     \label{formulavariacionalcomintegral1}
\end{equation}
The last equality follows from the existence, as we saw above, of a non-unique $n$-form
$\Psi$ such that $\Pi=\theta\wedge\Psi$. The conclusion is that
\begin{equation}\label{stationarysubmanifold}
 \left.\frac{\dx}{\dx t}\right\vert_0\calF_\Lambda(N_t)=0\qquad\mbox{if and only if}\qquad f^*\Psi=0  .
\end{equation}
A Legendre submanifold satisfying \eqref{stationarysubmanifold} is called a
\textit{stationary} Legendre submanifold. The exterior differential system algebraically
generated by $\theta,\dx\theta,\Psi$ is called the \textit{Euler-Lagrange system} of
$(S,\theta,\Lambda)$; its Poincar\'e-Cartan form $\Pi$ is said to be non-degenerate if
it has no other degree 1 factors besides the multiples of $\theta$.

In sum, the guiding rule to determine the critical submanifolds of
\eqref{Lagrangianfunctional} is the computation of the Poincar\'e-Cartan form
\eqref{PoncareCartanform}, its transformation into the product $\theta\wedge\Psi$ and
finally, due to \eqref{stationarysubmanifold}, the analysis of condition
$f^*\Psi=0$, that is, the Euler-Lagrange equation.

\subsection{Euler-Lagrange systems on $S_1M$ and applications}
\label{ELsottsb}

Again we consider an oriented $n+1$-dimensional Riemannian manifold
$M$ together with its unit tangent sphere bundle $\tsbu\stackrel{\pi}\lrr M$,
endowed with the canonical Sasaki metric and metric connection, possibly with torsion,
$\na^\divideontimes=\na^*-\frac{1}{s^2}\langle\na^*\ \,,\xi\rangle\xi$ induced from the Levi-Civita connection on $M$.

For the rest of this section we assume the notation $f:N\rr M$ refers to a compact
isometric immersed oriented submanifold of $M$ of dimension $n$.

% For simplicity we assume $N$ is oriented, but in regard to the problems below this may be overcome by passing to a double cover.

There exists a smooth lift $\hat{f}:N\rr\tsbu$ of $f$. We simply
define $\hat{f}(x)=\vec{n}_{f(x)}$, the unique unit normal in $T_{f(x)}M$ chosen
according to the orientations of $N$ and $M$. Note that $\hat{f}$ is also defined on
$\partial N$. It is easy to see that, up to the vector bundle isometry
$\dx\pi_{|H}:H\rr\pi^*TM$ on the horizontal side, we have the decomposition into
horizontal plus vertical:
\begin{equation}\label{derivdecomp}
 \dx \hat{f}(w)=\dx f(w)+(f^*\na)_wf^*\vec{n}\ ,\ \,\forall
w\in T_xN .
\end{equation}
Indeed, due to \eqref{pullbackconnectionproperties}, the vertical part at each point
$x\in N$ is $\na^*_{\dx{\hat{f}}(w)}\xi=({\hat{f}}^*\na^*)_w{\hat{f}}^*\xi$, where $\xi$ is the canonical vertical vector field on $\tsbu$. Clearly,
$({\hat{f}}^*\xi)_x=\hat{f}(x)=\vec{n}_{f(x)}=(f^*\vec{n})_x$.

By definition of $\hat{f}$ and \eqref{derivdecomp} we have that $\hat{f}:N\rr\tsbu$ defines a
Legendre submanifold of the natural contact structure: $\hat{f}^*\theta=0$. In other
words, $\hat{f}(N)$ is an integral submanifold of $\theta$ and $\dx\theta$. If we
choose an adapted direct orthonormal coframe $e^0,e^1,\ldots,e^{2n}$ locally on
$\tsbu$, then it may not be tangent to $N\hookrightarrow\tsbu$. Yet we have also a direct
orthonormal coframe $e^1,\ldots,e^{n}$ for $N$ (we use the same letters for the
pullback). Now, from \eqref{derivdecomp}, for any $1\leq j\leq n$ we have
\begin{equation}\label{pulbacksdeadaptedcoframeum}
 {\hat{f}}^*e^{j}=e^j
\end{equation}
and
\begin{equation}\label{pulbacksdeadaptedcoframedois}
 {\hat{f}}^*e^{j+n}=-\sum_{k=1}^nA^j_{k}e^k
\end{equation}
with $A$ the second fundamental form of $N$. We recall, $A=-\na\vec{n}:TN\rr TN$.

(By Proposition \ref{dmu} and ${\hat{f}}^*\dx\theta=\dx{\hat{f}}^*\theta=0$  we discover, or rather confirm, that $A^j_{k}$ is a symmetric tensor.)

Conversely, a smooth Legendre submanifold $Y$ is locally the lift $N\rr
Y\subset\tsbu$ of an oriented smooth $n$-submanifold $N\hookrightarrow M$ if and
only if ${e^{1\cdots n}}_{|Y}\neq0$, for this guarantees that $Y$ is transverse. We are thus going to assume throughout such \textit{open} condition on submanifolds, defined by the ground differential form, ${\alpha_0}_{|Y}\neq0$.

Let us now recall a well-known result on minimal surfaces. Our proof is essentially
the generalisation of the flat Euclidean case seen in \cite{BGG} (using the
Grassmannian bundle of $n$-planes for the supporting fibre bundle over $M$). We show
the proof works with any Riemannian manifold $M$. Moreover, one easily writes the
formula for the first-variation of the volume with fixed boundary, cf. \cite{Xin}.
\begin{teo}[Classical Theorem]
Let $N$ be a compact isometrically immersed hypersurface in $M$ and let $H$ be the
mean curvature vector field, i.e. 
$H=\tr{A}\,\vec{n}$. Then,  $\forall
v\in\Gamma_0(N,f^*TM)$,
\begin{equation}
 {\boldsymbol{\delta}}\vol(N)(v)=-\int_N\langle v,H\rangle\,\vol_N    .
\end{equation}
In particular, $N$ is stationary for the volume functional within all compact
hypersurfaces with fixed boundary $\partial N$ if and only if its mean curvature
vanishes.
\end{teo}
\begin{proof}
By \eqref{pulbacksdeadaptedcoframeum} we see immediately that the volume of an
$n$-dimensional submanifold $N$ in $M$ is given by
$\calF_{\alpha_0}(N)$, as in \eqref{Lagrangianfunctional}. Now, by Theorem \ref{derivadasdasnforms},
the Poincar\'e-Cartan form of $\alpha_0$ is just
$\dx\alpha_0=\theta\wedge\alpha_{1}=\Pi$. Thus a Legendre submanifold with fixed
boundary in $\tsbu$ is stationary for $\calF_{\alpha_0}$ if and only if
$\hat{f}^*{\alpha_{1}}=0$. Since 
\[ \alpha_{1}=e^{1+n}\wedge e^2\wedge\cdots\wedge e^n+e^1\wedge e^{2+n}\wedge
e^3\wedge\cdots\wedge  e^{n}+\: \text{etc} , \]
we have $\hat{f}^*{\alpha_{1}}=-(A^1_{1}+\cdots+A^n_{n})\,e^1\wedge\cdots\wedge
e^n=-\langle H,\vec{n}\rangle\,\vol_N$. Then we apply
\eqref{formulavariacionalcomintegral1}, noticing that $v\lrcorner {\hat{f}}^*\theta=\langle
v,\vec{n}\rangle$.
\end{proof}
We now consider the other $n$-forms $\alpha_i$ defined in Section
\ref{Thenewnforms}. They give in their own right interesting Lagrangian systems on
the contact manifold $\tsbu$. We define the functionals on the set of
compact immersed hypersurfaces of $M$:
\begin{equation}
 \calF_i(N):=\int_N\alpha_i.
\end{equation}

Let ${\sympol}_i(A)$ denote the elementary symmetric polynomial of degree $i$ on the eigenvalues $\lambda_1,\ldots,\lambda_n$ of $A$.

The proof of the following result is quite immediate and in line with the deduction of previous identities, cf. Section \ref{PfsELsystems}.
\begin{prop}\label{pullbackhatalphai}
For $0\leq i\leq n$, we have $\hat{f}^*\alpha_{i}=(-1)^i\sympol_i(A)\,\vol_N$.
\end{prop}
The Lagrangians referred in \cite[p. 32]{BGG}, defined over the
Grassmannian bundle of $n$-planes over flat Euclidean space, call our attention to
the following feature on the general setting of a Riemannian manifold $M$. We have the functional $\calF_n$ satisfying
\begin{equation}\label{IntegralofGaussKroneckercurvature}
 \calF_n(N)=(-1)^n\int_N \det(A)\,\vol_N  .
\end{equation}
$\det(A)$ is the so-called \textit{Gauss-Kronecker curvature} of $N$ when $M$ is
flat. Moreover, by \eqref{Ralphan}, if $M$ is flat we have $\dx\alpha_n=0$ and hence
a functional constant on the set of $\mathrm{C}^1$-close
isometrically immersed $N$, which partly confirms the Theorem of Chern-Gauss-Bonnet.

This result may be immediately generalised recurring to \eqref{dalphaicomcsc}. For example, suppose $M$ has constant sectional curvature $c\neq0$. Then $\dx\alpha_n=-c\,\theta\wedge\alpha_{n-1}$ and so the Gauss-Kronecker functional with fixed boundary admits $N$ as a stationary submanifold if and only if $\lambda_2\lambda_3\cdots\lambda_n+ \lambda_1\lambda_3\cdots\lambda_n+ \cdots
 +\lambda_1\lambda_2\cdots\lambda_{n-1}=0$. Equivalently,
\begin{equation}
 \exists j,k,\ j\neq k:\  \lambda_j=\lambda_k=0
 \quad  \ \mbox{or}\ \quad   \sum_{i=1}^n\frac{1}{\lambda_i}=0.
\end{equation}

As seen earlier, the functional $\calF_{1}$ corresponds with
\begin{equation}
 \calF_{1}(N)=-\int_N\langle H,\vec{n}\rangle\,\vol_N ,
\end{equation}
basically the integral of the mean curvature on immersed $f:N\rr M$.
\begin{teo} \label{teoremamenacurvaturefunctional}
Suppose $M$ has dimension $n+1\geq 3$. Then any compact isometric
immersed hypersurface ${f}:N\rr M$ satisfies, $\forall v\in\Gamma_0(N,f^*TM)$,
\begin{equation}
 {\boldsymbol{\delta}}\calF_1(N)(v)=\int_N\langle v,\vec{n}\rangle
 \bigl(r_{\vec{n}}-\mathrm{Scal}^M+\mathrm{Scal}^N\bigr)\vol_N    .
\end{equation}
Hence $N$ is stationary for the mean curvature functional
$\calF_{1}$ with fixed boundary $\partial N$ if and only if 
\begin{equation}
 \mathrm{Scal}^N=\mathrm{Scal}^M-r_{\vec{n}}
\end{equation}
where $r_{\vec{n}}=\ric(\vec{n},\vec{n})$ is induced from the Ricci tensor of $M$ and $\mathrm{Scal}^N,\mathrm{Scal}^M$ are the respective scalar curvature functions.

In particular, if the metric on $M$ is Einstein, say $\ric=nc\langle\ ,\ \rangle$, then $N$ has stationary integral of the mean curvature with fixed boundary if and only if $N$ has constant scalar curvature $\mathrm{Scal}^N=n^2c$.
\end{teo}
\begin{proof}
By Gauss and Codazzi equations, we see the curvatures of $N$ and $M$ satisfy $R^M_{ijij}=R^N_{ijij}-\lambda_i\lambda_j,\
\,\forall1\leq i,j\leq n$, in an orthonormal basis diagonalising the second
fundamental form $A$. Hence
\begin{eqnarray*}
 \mathrm{Scal}^M &=& \sum_{i,j=0}^nR^M_{ijij} \\
    &=&   2\sum_{j=0}^nR^M_{j0j0}+ \sum_{i,j=1}^nR^M_{ijij}  \\
    &=&  2r_{\vec{n}}+\mathrm{Scal}^N-2\sympol_2(A) ,
\end{eqnarray*}
which is actually a known formula, since $\sympol_2(A)=
\sum_{i<j}\lambda_i\lambda_j$. Then, recurring to the fundamental exterior differential system on $\tsbu\rr M$, by \eqref{dalphaum} we
know $\dx\alpha_{1}=\theta\wedge(2\alpha_{2}-r\,\alpha_0)$, which is the Poincar\'e-Cartan form, where $r$ was defined in \eqref{defofr}. Hence, by \eqref{formulavariacionalcomintegral1}, we have
\begin{equation*}
  {\boldsymbol{\delta}}\calF_1(N)(v) =\int_N(v\lrcorner {\hat{f}}^*\theta)\,{\hat{f}}^*(2\alpha_2-r\,\alpha_0)= \int_N\langle v,\vec{n}\rangle
 \bigl(2\sympol_2(A)-r_{\vec{n}}\bigr)\vol_N .
\end{equation*}
Indeed ${\hat{f}}^*r$ in the previous notation agrees with $r_{\vec{n}}$. The first part follows. For the second, the formula yields $\mathrm{Scal}^N=(n+1)nc-nc=n^2c$.
\end{proof}

We recall a similar known result with the hypothesis of Euclidean ambient, however within a different variational principle. It is the following: constant scalar curvature hypersurfaces $N$ of Euclidean space $\R^{n+1}$ are stationary for $\calF_{1}$, when varying within the class of volume preserving submanifolds. 

The theory yields partly a classical result.
\begin{teo}[cf. {\cite[Theorem B]{Rei}}]    \label{teo_partofTheoremBofRei}
Suppose $M$ has constant sectional curvature $c$. A compact isometric immersed hypersurface $N$ satisfies, for all $0\leq i\leq n$ and $v\in\Gamma_0(N,f^*TM)$,
\begin{equation}
 {\boldsymbol{\delta}}\calF_i(N)(v)=(-1)^i\int_N\langle v,\vec{n}\rangle
 \bigl(c(n-i+1)\sympol_{i-1}(A)-(i+1)\sympol_{i+1}(A)\bigr)\vol_N .   
\end{equation}
$N$ is stationary for $\calF_i$ if and only if
$c(n-i+1)\sympol_{i-1}(A)-(i+1)\sympol_{i+1}(A)=0$.
\end{teo}
\begin{proof}
 Immediate from formula \eqref{dalphaicsc}.
\end{proof}

Now, for an Einstein metric on the ambient manifold $M$, we see through a formula in the proof of Theorem \ref{teoremamenacurvaturefunctional} that $\calF_{2}$ leads to an Euler-Lagrange equation on the scalar curvature of $N$.
\begin{teo}
Let $M$ be a Riemannian manifold of dimension $n+1>2$ and constant sectional curvature $c$. Then a compact hypersurface $N$ is a critical point of the scalar curvature functional $\int_N\mathrm{Scal}^N\,\vol_N$ with fixed boundary if and only if the eigenvalues $\lambda_1,\ldots,\lambda_n$ of $A$ satisfy (assume $\lambda_3=0$ for $n=2$)
\begin{equation}
 \sum_{j_1<j_2<j_3}\lambda_{j_1}\lambda_{j_2}\lambda_{j_3}
                        +\frac{c}{6}(n-1)(n-2)(\lambda_1+\cdots+\lambda_n)=0    .
\end{equation}
That is, $6\sympol_3(A)+c(n-1)(n-2)\sympol_1(A)=0$.
\end{teo}
\begin{proof}
As seen above, we have $\mathrm{Scal}^N=\mathrm{Scal}^M-2r_{\vec{n}}+2\sympol_2(A)$ corresponding here to the Lagrangian $\Lambda=(n+1)nc\,\alpha_0-2nc\,\alpha_0+2\alpha_{2}=(n-1)nc\,\alpha_0+2\alpha_2$. Recurring to \eqref{dalphaicsc} we compute
\begin{eqnarray*}
 \dx\Lambda &=& (n-1)nc\,\theta\wedge\alpha_{1}+
                   2\theta\wedge\bigl(3\alpha_{3}-c(n-1)\,\alpha_{1}\bigr)\\
            &=& \theta\wedge\bigl(6\alpha_{3}+c(n-1)(n-2)\,\alpha_{1}\bigr)
\end{eqnarray*}
and thus, having found the Poincar\'e-Cartan form, the result follows easily.
\end{proof}

Note the case $n=2$ is always satisfied and invariant of the ambient manifold as expected
by Gauss-Bonnet Theorem.

For $n=1$, we remark $\calF_{0}$ gives the unparametrised geodesics as length stationary submanifolds and $\calF_1$ gives a trivial condition. The following functional, with $t\in\R$, seems
also particularly interesting for further studies with our system
\[ \calF(t,N)=\sum_{i=0}^n\int_Nt^i\alpha_{n-i}=\int_N\det{(t1-A)}\,\vol_N   .  \]

The celebrated integral identities of Hsiung-Minkowski (\cite{Hsiung,Katsurada}) are easy to deduce using our fundamental exterior differential system. Following their notation, we define $M_i=\sympol_i(A)/\binom{n}{i}$.
\begin{teo}[Hsiung-Minkowski identities]  \label{teo_HsiungMinkowski}
Let $N$ be any given closed oriented immersed hypersurface of Euclidean space $\R^{n+1}$. Let $\vec{n}$ be the unit normal to $N$ with the induced orientation. And let $X$ be the position-vector vector field (from the origin). Then, for any $0\leq i\leq n$,
\begin{equation}
 \int_N(M_i+\langle X,\vec{n}\rangle M_{i+1})\,\vol_N=0.
\end{equation}
\end{teo}
The proof of this theorem with the fundamental exterior differential system is much simpler than the original. We leave it as an exercise. Notice the result is not true in wider contexts, since the ${\cal L}_{X^h}\alpha_i$ are then not so easy to compute. The case is that the lift of $X$, even for Euclidean space, is no longer a Killing vector field of the Sasaki metric.

It is in \cite{Katsurada} that we see the formula for any $i$, attributed to Hsiung. Of course, if $\partial N\neq\emptyset$, then a more general result also follows.

\subsection{Infinitesimal symmetries}

We now wish to explore further properties of the exterior differential system on the
tangent sphere bundle of a Riemannian $n+1$-manifold. We resume with previous
notation for $\tsb$ with any radius $s$. For many of the following notions we recur
to \cite{BCGGG}, \cite{BGG} and \cite{IveyLan}.

It is easy to see there are no non-zero Cauchy characteristics of the contact structure
$(\tsb,\theta)$, i.e. there exists no vector field $v\in\XIS_\tsb\backslash0$ on
$\tsb$ such that $v\lrcorner\calI\subset\calI$ where $\calI$ is the $\dx$-closed
differential ideal generated by $\theta$. The Lie algebra $\g_\calI$ of infinitesimal
symmetries $v$ of $\calI$, a set containing the Cauchy characteristics, is easy to compute
formally. $v$ is now required to satisfy ${\cal L}_v\calI\subset\calI$. On an adapted
frame on $\tsb$, we let $v=\sum_{i=0}^{2n}v_ie_i$ and it is of course enough to check
that the 1-form ${\cal L}_v\theta$ is still a multiple of $\theta$. By the Cartan
formula,
\begin{equation}
 {\cal L}_v\theta=\dx(v\lrcorner\theta)+v\lrcorner\dx\theta\in\calI\ \
\Longleftrightarrow\ \left\{\begin{array}{l}
                             s\,\dx v_0(e_i)=-v_{i+n} \\ s\,\dx v_0(e_{i+n})=v_{i}
                            \end{array}, \right. \forall0\leq i\leq n   . 
\end{equation}
In particular $B^{\mathrm{t}}\xi=\theta^\sharp$, the geodesic spray vector
field on $\tsb$,
is an infinitesimal symmetry. Recall from \eqref{igualdade1} that $\Omega^j_\tsb\subset\calI,\ \forall j>n$, where $\Omega_\tsb^j$ is the space of $j$-forms. Hence, for each $i$, the Euler-Lagrange system spanned by $\{\theta,\dx\theta,\alpha_i\}$ forms a $\dx$-closed ideal. 

Let $\calI_{-1}=\calI$ and let $\calI_i$ be the ideal generated by
$\calI_{i-1}\cup\{\alpha_{n-i}\}$, for each $0\leq i\leq n$. Clearly
$\calI_{i}\subset\calI_{i+1}$ and $\calI_n$ agrees with the fundamental differential system.

Now we must return to $M$ with constant sectional curvature.

We have both a $\dx$-closed ideal filtration, such that $\dx\calI_i\subset\calI_{i+1}$, and a \textit{Lie filtration}, cf. \eqref{dalphaicsc}:
\begin{equation}
 {\cal L}_{\theta^\sharp}\calI_i\subset\calI_{i+1}  .
\end{equation}
One considers the tautological horizontal vector field in the class of constant sectional
curvature metrics, since, even in such case and choosing any \textit{constant}
Lagrangian $n$-form $\Lambda$ from the differential system $\calI_n$, it is barely decidable if it admits non-trivial infinitesimal symmetries.

We want a $\Lambda=\sum_{i=0}^nX_i\alpha_i$ with real constant coefficients $X_i$
generating a $\dx$-closed ideal 
${\cal J}=\{\theta,\dx\theta,\Lambda\}\subset\calI_n$, and the goal is to
guarantee $\theta^\sharp\in\g_{\cal J}$. Letting $c$ denote the sectional curvature,
then we compute:
\begin{equation}\label{Liederivativebygeodesicspray}
  {\cal L}_{\theta^\sharp}\Lambda=\theta^\sharp\lrcorner\dx\Lambda=
\sum_{i=0}^ns^2X_i(\frac{1}{s^2}(i+1)\alpha_{i+1}-c(n-i+1)\alpha_{i-1})   .
\end{equation}
No component of this is a multiple of $\dx\theta$ because there are no
$e^{j(j+n)}$ factors in any of the terms of any of the $\alpha_i$. So it can only be a multiple $\frac{a}{s^2}\Lambda$ of $\Lambda$ itself for some real function $a$, and hence this becomes
\begin{equation}
 \sum_{j=0}^n(jX_{j-1}-s^2c(n-j)X_{j+1}-aX_j)\alpha_j=0    .
\end{equation}
In particular $a$ is a constant. Let us put the coefficients in linear system $LX=0$. Then the determinant of the $n+1$-squared matrices representing $L$, the first three being
\begin{equation}
 \left[\begin{array}{cc} -a& -{cs^2}\\ 1& -a         \end{array}\right]\qquad
\left[\begin{array}{ccc} -a& -2{cs^2} & \\ 1 &  -a& -{cs^2} \\ &2& -a    
\end{array}\right]\qquad
\left[\begin{array}{cccc} -a& -3{cs^2} &  & \\ 1 & -a& -2{cs^2} & \\ &2&-a& -{cs^2} \\ &
&3&-a  \end{array}\right]   ,
\end{equation}
is
\begin{equation}
 \det{L}=\left\{\begin{array}{ll}
 \ \ (a^2+{cs^2})(a^2+9{cs^2})(a^2+25{cs^2})\cdots(a^2+n^2{cs^2}) & \mbox{for
$n$ odd}  \\
 -a(a^2+4{cs^2})(a^2+16{cs^2})(a^2+36{cs^2})\cdots(a^2+n^2{cs^2}) & \mbox{for
$n$ even} \end{array} \right. .
\end{equation}
Since the vanishing of this determinant assures the non-trivial solutions, the conclusions are as follows.
\begin{teo}
Let $M$ have constant sectional curvature $c$ and let $\Lambda$ be defined, as above, on the tangent sphere bundle $\tsb$. Then the condition $\theta^\sharp\in\g_{\cal J}$ is equivalent to solving ${\cal L}_{\theta^\sharp}\Lambda=a\Lambda$ for constant $a$. Moreover\\
i) in case $n$ is even or $M$ is flat ($c=0$), there always exists a
1-dimensional solution subspace $\R\Lambda$, hence satisfying $a=0$;\\
ii) in case $n$ is odd and $c>0$, then there is no non-trivial solution;\\
iii) in case $c\leq0$, then there exist $n+1$ solution subspaces $\R\Lambda_j$
satisfying 
\begin{equation}
 {\cal L}_{\theta^\sharp}\Lambda_j=a_j^\pm\Lambda_j\ \,\mbox{with}\,\
\left\{\begin{array}{ll}
a_j^\pm=\pm(2j+1)s\sqrt{-c},\ \forall0\leq j\leq \frac{n-1}{2} & \mbox{for $n$ odd}\\
a_j^\pm=\pm(2j)s\sqrt{-c},\ \forall0\leq j\leq \frac{n}{2} & \mbox{for $n$ even}
\end{array}   \right.   .
\end{equation}
Furthermore, whenever $a=0$, we have $\dx\Lambda=0$. 
\end{teo}
The last assertion follows from \eqref{Liederivativebygeodesicspray}. The theory requires further a study of the conservation laws of the Euler-Lagrange system $\{\theta,\dx\theta,\Lambda\}$. Besides the trivial ones, $\theta^\sharp\lrcorner\Lambda$ and $\delta\Lambda$, which in fact vanish by Proposition \ref{dasteriscoalphaicsc}, we question if there may exist any other.

\section{Proofs of main formulae}
\label{Pomf}

\subsection{An algebraic technique}
\label{Auat}

We start by recalling an algebraic tool which creates new differential forms from tensors on a given manifold. Such technique was introduced in \cite{AlbSal1} and the proofs of all assertions regarding it are quite straightforward.

Given any $p$-tensor $\eta$ and any endomorphisms $B_i$, $1\leq i\leq p$, of the tangent bundle of the given manifold, we let
\begin{equation}
\eta\circ (B_1\wedge\cdots\wedge B_p) 
\end{equation}
denote the $p$-form defined by ($S_p$ is the symmetric group)
\begin{equation}\label{esttensorcont}
\eta\circ (B_1\wedge\cdots\wedge B_p)(v_1,\ldots,v_p)=\sum_{\sigma\in
S_p}\mathrm{sg}(\sigma)\,\eta(B_1v_{\sigma_1},\ldots,B_pv_{\sigma_p}) .
\end{equation}

If $\eta$ is a $p$-form, then $\eta\circ(\wedge^p1)=p!\,\eta$. For a wedge of $p$
1-forms we have the most important identities:
\begin{equation}  \label{techniqueofdifferentialforms}
\begin{split}
\eta_1\wedge\ldots\wedge\eta_p\circ(B_1\wedge\ldots\wedge B_p)\ =\ \sum_{\sigma\in S_p} 
\eta_1\circ B_{\sigma_1}\wedge\ldots\wedge\eta_p\circ B_{\sigma_p} \ \ \ \\
=\ \sum_{\tau\in S_p} \mathrm{sg}({\tau})\:\eta_{\tau_1}\circ
B_{1}\wedge\ldots\wedge\eta_{\tau_p}\circ B_{p} .
\end{split}
\end{equation}
For a 2-form $\eta$ and any endomorphism $B$, clearly \,$\eta\circ B\wedge
B\,(v,w)=2\eta(Bv,Bw)$. For a 3-form and two endomorphisms $B,C$, letting
$\cyclic$ denote cyclic sum, we have
\begin{equation}
\eta\circ(B\wedge B\wedge C)(v,w,z)=2\cyclic_{v,w,z}\eta(Bv,Bw,Cz) .
\end{equation} 

A great advantage of the construction of forms as above is that it obeys a simple Leibniz rule under either Lie or covariant differentiation, with no minus signs attached:
\begin{equation}  \label{derivacao}
\begin{split}
 \calD(\eta\circ(B_1\wedge\cdots\wedge B_p)) 
=(\calD\eta)\circ (B_1\wedge\cdots\wedge B_p) \qquad \\
 +\sum_{j=1}^p\eta\circ(B_1\wedge\cdots\wedge\calD B_j\wedge\cdots\wedge B_p) .
\end{split}
\end{equation}

\subsection{Proofs for Section \ref{Geomofthetangentspherebundle}}
\label{PfSGeom}

We resume with the theoretical setting of Section \ref{TnedsoS}. We are now ready for the proofs of the auxiliary and main formulae there. The theory recurs to the now classical geometry of tangent bundles initiated by Sasaki (cf. \cite{Alb5,Alb3,Alb4,Blair,Sakai,Tash}).

To encourage the reading which follows, we start by giving yet another explanation of
the well-known formulae \eqref{pullbackconnectionproperties}. In a way, these are the
defining equations of a connection in relation with the horizontal subspace $H$. We
may always assume a non-vertical vector is of the form $w=(\dx v_1)_x(v_2)$ where
$x\in M$ and $v_1,v_2\in\XIS_M$ are vector fields. We just need the map $v_1$ into
$TM$ to be defined on a neighbourhood of $x$. As it is not so difficult to see,
$w^v=((\dx v_1)_x(v_2))^v=\na_{v_2}v_1$. Then since $\pi\circ v_1=1_M$ and
$v_1^*\xi_x=\xi_{v_1(x)}=v_1(x)$, we find
\[ \na^*_w\xi=(v_1^*\pi^*\na)_{v_2}v_1^*\xi=\na_{v_2}v_1=w^v  .   \]
Furthermore, if one considers a curve $\gamma$ in $M$ and its velocity and acceleration, then one has that $\ddot{\gamma}$ sits in the sub-bundle $\ker(\na^*_\cdot\xi)$ if and only if the curve satisfies the equations system of a geodesic of $M$.

\begin{proof}[Proof of (\ref{torsaodenablaasterisco}) and Proposition \ref{dmu}]
Recall the connection $D$ on $\tsb\subset TM$ induced from $\na^*$, given in Section
\ref{Geomofthetangentspherebundle}: \,$D=\na^\divideontimes-\frac{1}{2}\calR$. To prove it is torsion-free, we may likewise compute the torsion of $\na^\divideontimes$. First, it is easy to see that $(T^{\na^\divideontimes})^h=(T^{\na^*})^h=\pi^*T^\nabla=0$. Secondly, disregarding the symmetric component $\na^\divideontimes_{y}z-\na^*_yz=\frac{1}{s^2}\langle y^v,z^v\rangle\xi$, for any two vector fields $y,z$ on
$\tsb$, cf. \eqref{nablastar}, the vertical part is 
\begin{eqnarray*}
(T^{\na^\divideontimes}(y,z))^v &=& \na^\divideontimes_{y}z^v-\na^\divideontimes_{z}y^v -[y,z]^v\\
&=& \na^*_{y}\na^*_{z}\xi-\na^*_{z}\na^*_{y}\xi -\na^*_{[y,z]}\xi\\
&=& \calR(y,z) .
\end{eqnarray*}
This proves (\ref{torsaodenablaasterisco}). Regarding the 1-form $\theta$ defined in
(\ref{mu}), we use the same connection to compute firstly:
\begin{eqnarray*}
 (D_y\theta)z &=& y(\theta(z))-\theta(D_yz) \\
&=&y\langle\xi,Bz\rangle-\langle\xi,BD_yz\rangle\\
&=&\langle\na^*_y\xi,Bz\rangle+\langle
\xi,\na^*_yBz\rangle-\langle\xi,B\na^*_yz\rangle  \\
&=&\langle y^v,Bz\rangle\ .
\end{eqnarray*}
Hence, by a well-known formula,
\[  \dx\theta(y,z)\ =\ (D_y\theta)z-(D_z\theta)y\ =\ \langle y,Bz\rangle-\langle
z,By\rangle \]
as we wished.
\end{proof}

\subsection{Proofs for Section \ref{Thenewnforms}}
\label{PfSThenewnforms}

Note that $\theta$ is actually defined on the manifold $TM$. Undoubtedly $\theta$ corresponds with
the pullback of the Liouville 1-form on the cotangent bundle through the
musical isomorphism induced by the metric. Hence $\dx\theta$ corresponds with the
pullback of the canonical, exact symplectic form of $T^*M$ (cf. the general case of
a connection with torsion in \cite{Alb3}). Using the adapted direct orthonormal frame
$\{e_0,e_1,\ldots,e_n,e_{n+1},\ldots,e_{2n}\}$, locally defined on $\tsb$, we prove
the basic structure equations.
\begin{proof}[Proof of Proposition \ref{Basicstrutequations}]
First note (recall the notation $e^{ab}=e^a\wedge e^b$)
\begin{eqnarray*}
(\dx\theta)^i &=&
\sum_{j_1=1}^ne^{(n+j_1)j_1}\,\wedge\,\cdots\,\wedge\,\sum_{j_i=1}^ne^{(n+j_i)j_i}\\
& = &  \sum_{1\leq j_1<\ldots<j_i\leq n}
i!\,e^{(n+j_1)j_1}\wedge\cdots\wedge e^{(n+j_i)j_i}  .
\end{eqnarray*}
In particular, one proves the claim that $(\dx\theta)^n=
(-1)^{\frac{n(n+1)}{2}}n!\,e^{1\ldots n(n+1)\ldots(2n)}$, enough to ensure we
have a contact structure. Now 
\begin{eqnarray*}
 *(\dx\theta)^i
&=& i!(-1)^{\frac{n(n+1)}{2}} \sum_{1\leq k_1<\ldots<k_{n-i}\leq n}e^0\wedge
e^{(n+k_1)k_1} \wedge\cdots\wedge e^{(n+k_{n-i})k_{n-i}} \\
&=& (-1)^{\frac{n(n+1)}{2}} \frac{i!}{(n-i)!s}\,\theta\wedge(\dx\theta)^{n-i} .
\end{eqnarray*}
This proves the first part of (\ref{Basicstrutequations2}) and, in particular, (\ref{Basicstrutequations1}) due to $**=1_{\Lambda^*}$.

Now, applying the second identity of (\ref{techniqueofdifferentialforms}), we find
\begin{eqnarray*}
 \alpha_i &=& n_i\,\alpha_n\circ(B^{n-i}\wedge1_{T\tsb}^{i}) 	\\  &=& n_i\sum_{\sigma\in S_n}\mathrm{sg}(\sigma)\bigl(e^{(n+\sigma_1)}\circ
  B\wedge\cdots\wedge e^{(n+\sigma_{n-i})}\circ B\wedge 
   e^{(n+\sigma_{n-i+1})}\wedge\cdots\wedge e^{(n+\sigma_n)}\bigr)  \\
 &=& n_i\sum_{\sigma\in S_n}\mathrm{sg}(\sigma)\bigl(e^{\sigma_1}\wedge\cdots\wedge
e^{\sigma_{n-i}}\wedge e^{(n+\sigma_{n-i+1})}\wedge\cdots\wedge e^{(n+\sigma_n)}\bigr)  .
\end{eqnarray*}
Since $n_i=n_{n-i}$, we find
\begin{eqnarray*}
 *\alpha_{n-i}
&=& n_i\sum_{\sigma}\mathrm{sg}(\sigma)*\bigl(e^{\sigma_1}\wedge\cdots\wedge
e^{\sigma_{i}}\wedge e^{(n+\sigma_{i+1})}\wedge\cdots\wedge
e^{(n+\sigma_n)}\bigr)\\
&=& n_i\sum_{\sigma}\mathrm{sg}(\sigma)\,(-1)^{n+(n-i)n}\,e^0\wedge
e^{\sigma_{i+1}}\wedge\cdots\wedge e^{\sigma_n}\wedge
e^{(n+\sigma_1)}\wedge\cdots\wedge e^{(n+\sigma_{i})}\\
&=& \frac{n_i}{s}\sum_{\tau}\mathrm{sg}(\tau)\,(-1)^{in+i(n-i)}\,\theta\wedge
e^{\tau_{1}}\wedge\cdots\wedge e^{\tau_{n-i}}\wedge
e^{(n+\tau_{n-i+1})}\wedge\cdots\wedge e^{(n+\tau_{n})}\\
&=& \frac{(-1)^{i}}{s}\,\theta\wedge\alpha_{i} ,
\end{eqnarray*}
where the $\tau$ equal the $\sigma$ composed with an obvious index
permutation. Formulae $\dx\theta\wedge\alpha_i=0,\
\alpha_i\wedge\alpha_j=0,\ \forall j\neq n-i$, are then very easy to deduce.
\end{proof}
Now let us see the proof of an important result.
\begin{proof}[Proof of Theorem \ref{derivadasdasnforms} in Section \ref{Thenewnforms}]
Recall, for all $0\leq i\leq n$,
\begin{equation*}
 \alpha_i=n_i\,\alpha_n\circ(B^{n-i}\wedge1^{i})
\end{equation*}
where $\alpha_n=\frac{\xi}{\|\xi\|}\lrcorner\inv{\pi}\vol_M$ requires the
vertical pullback and 1 denotes the identity endomorphism of $T\tsb$.

With the torsion-free linear connection $D$ on $\tsb$ and with the adapted frame
$\{e_0,\ldots,e_{2n}\}$ together with its dual coframe, we are well equipped to compute a formula for $\dx\alpha_i$. It is obtained through the
well-known general
formula $\dx\alpha_i=\sum_je^j\wedge D_j\alpha_i$. Hence we need the following
computation: $\forall v,v_1,\ldots,v_n$ vector fields on $\tsb$,
\begin{equation}\label{D_valpha_i}
\begin{split}
 \lefteqn{D_v\alpha_i(v_1,\ldots,v_n) \ =}\\
&\ = v\cdot(\alpha_i(v_1,\ldots,v_n))-
\sum_{k=1}^n\alpha_i(v_1,\ldots,\na^\divideontimes_vv_k-\frac{
1}{2}\calR(v,v_k),\ldots,v_n)\\
&\ =\na^*_v\alpha_i(v_1,\ldots,v_n) 
+\frac{1}{2}\sum_k\alpha_i(v_1,\ldots,\calR(v,v_k),\ldots,v_n)\\
&\ = n_i(\na^*_v\alpha_n)\circ(B^{n-i}\wedge1^{i})(v_1,\ldots,v_n)+
 \frac{1}{2}\sum_k\alpha_i(v_1,\ldots,\calR(v,v_k),\ldots,v_n) .
\end{split}
\end{equation}
We have used the fact that any $\alpha_i$ vanishes in the direction of $\xi$ and that
$\na^*B=\na^*1=0$. Now, since $\na\vol_M=0$, it follows for both horizontal and
vertical lifts that $\na^*\pi^*\vol_M=\na^*\inv{\pi}\vol_M=0$. Then for any
$v_j\in\XIS_{\tsb}$, and since $v_j(\|\xi\|)=0$, we find
\begin{equation}\label{nablaasterisco_v_jalpha}
\begin{split}
\na^*_{v_j}\alpha_n&=\
\bigl(v_j\bigl(\frac{1}{\|\xi\|}\bigr)\xi+\frac{1}{\|\xi\|}\na^*_{v_j}
\xi\bigr)\lrcorner(\inv{\pi}\vol_M)+\frac{\xi}{\|\xi\|}\lrcorner(\na^*_{v_j}\inv
{\pi}\vol_M)\\
&=\ \frac{1}{s}v_j^v\lrcorner(\inv{\pi}\vol_M) .
\end{split}
\end{equation}

Now one sees by (\ref{D_valpha_i}) and (\ref{nablaasterisco_v_jalpha}) that
$\dx\alpha_i$ has obvious \textit{flat} and \textit{curved} components. We compute
firstly such flat part of $\dx\alpha_i$. Proceeding by the mentioned formula,
\begin{equation*}
 \begin{split}
& \lefteqn{ \sum_{j=0}^{2n}e^j\wedge(\na^*_j\alpha_n)\circ(B^{n-i}\wedge1^{i})
\ =}\\
& \ \ \ =\ \frac{1}{s}\sum_{j=n+1}^{2n}
e^j\wedge\bigl((e_j\lrcorner\inv{\pi}\vol_M)\circ(B^{n-i}\wedge1^{i})\bigl)\\
 &\ \ \ =\ \frac{1}{s^2}\sum_{j=n+1}^{2n} (-1)^{j-n}\,
e^j\wedge\bigl((\xi^\flat\wedge
e^{(n+1)\cdots\hat{j}\cdots (2n)})\circ(B^{n-i}\wedge1^{i})\bigr)\\
&\ \ \ =\ \frac{1}{s^2}\sum_{j=1}^{n} (-1)^{j}\, e^{j+n}\wedge\bigl((\xi^\flat\wedge
e^{(n+1)\cdots \widehat{j+n}\cdots (2n)})\circ(B^{n-i}\wedge1^{i})\bigr) 
\quad=\ (\star).
 \end{split}
\end{equation*}
In the following step we define $B_1=\cdots=B_{n-i}=B$ and $B_{n-i+1}=\cdots=B_{n}=1$. By the first identity of formula (\ref{techniqueofdifferentialforms}), we have in
particular
\begin{equation}\label{alphaicombes}
 \alpha_n\circ(B^{n-i}\wedge1^{i})=\sum_{\sigma\in S_n} 
e^{n+1}\circ B_{\sigma_1}\wedge\cdots\wedge e^{2n}\circ B_{\sigma_n} .
\end{equation}
Applying the same technique from (\ref{techniqueofdifferentialforms}) in the computation above yields:
\begin{eqnarray*}
(\star)\ \ &=& \frac{1}{s^2}\sum_{j=1}^{n}\sum_{\sigma\in S_n} (-1)^{j}\,
e^{j+n}\wedge\xi^\flat\circ B_{\sigma_1}\wedge e^{n+1}\circ
B_{\sigma_2}\wedge\cdots  \\
& & \hspace{14mm}\cdots\wedge e^{j+n-1}\circ B_{\sigma_j}\wedge e^{j+n+1}\circ
B_{\sigma_{j+1}}\wedge\cdots\wedge e^{2n}\circ B_{\sigma_n}\\
&=& \frac{1}{s^2}\sum_{j=1}^{n}\sum_{\sigma\in S_n:\ \sigma_1\leq n-i}
\theta\wedge e^{n+1}\circ B_{\sigma_2}\wedge\cdots
\wedge e^{j+n-1}\circ B_{\sigma_j}\wedge e^{j+n}\wedge \\
& & \hspace{34mm} \wedge e^{j+n+1}\circ B_{\sigma_{j+1}}\wedge\cdots\wedge
e^{2n}\circ B_{\sigma_n}
\end{eqnarray*}
because $\theta=\xi^\flat\circ B$ and $\xi^\flat\circ1=0$. Now notice the role of $B_{\sigma_1}$. Letting $B_1=\cdots=B_{n-i-1}=B$ and $B_{n-i}=\cdots=B_{n}=1$, we may continue the computation:
\begin{eqnarray*}
&=& \frac{n-i}{s^2}\,\theta\wedge\sum_{j=1}^{n}\sum_{\tau\in S_n:\ \tau_j=n}
e^{n+1}\circ B_{\tau_1}\wedge\cdots
\wedge e^{j+n-1}\circ B_{\tau_{j-1}}\wedge e^{j+n}\wedge \\
& & \hspace{5cm}\wedge e^{j+n+1}\circ B_{\tau_{j+1}}\wedge\cdots\wedge
e^{2n}\circ B_{\tau_n}  \\ 
& = & \frac{n-i}{s^2}\,\theta\wedge\sum_{j=1}^{n}\sum_{\tau:\ \tau_j=n}
e^{n+1}\circ B_{\tau_1}\wedge\cdots
\wedge e^{j+n}\circ B_{\tau_{j}}\wedge \cdots\wedge e^{2n}\circ B_{\tau_n}\\
& =& \frac{n-i}{s^2}\,\theta\wedge\alpha_n\circ(B^{n-i-1}\wedge1^{i+1}) .
\end{eqnarray*}
Notice in case $i=0$ this expression vanishes because $\theta$ would have never
appeared. Henceforth, assuming for a moment that $M$ is flat, we have deduced, cf.
(\ref{dalphai}),
\begin{equation}
 \dx\alpha_i=\frac{n_i(n-i)}{s^2n_{i+1}}\,\theta\wedge\alpha_{i+1}
=\frac{i+1}{s^2}\,\theta\wedge\alpha_{i+1} .
\end{equation}

Now let us see the \emph{curved} side of (\ref{D_valpha_i}). First
recall that $\calR(v,\ )$ vanishes on any vertical direction $v$. Then, writing
as it is usual $R_{abcd}=\langle R^\na(e_c,e_d)e_b,e_a\rangle$, $\forall
a,b,c,d\in\{0,\ldots,n\}$ and using the adapted frame $e_0,\ldots,e_{2n}$ on $\tsb$, we have
\[  \calR_{e_j}=\sum_{p=1}^n\langle R^\na(e_j,\ )se_0,e_{p}\rangle e_{p+n} =
\sum_{q=0,\ p=1}^nsR_{p0jq}\,e^q\otimes e_{p+n}  \ . \]
It is easy to see we just have to simplify the following expression, coming from
(\ref{D_valpha_i}) and given for all $v_1,\ldots,v_n\in\XIS_{\tsb}$ by
\begin{equation*}
\begin{split}
\lefteqn{ \sum_{k=1}^n\alpha_i(v_1,\ldots,\calR_{e_j}v_k,\ldots,v_n)\ =\ }\\
&=
\sum_{q=0}^n\sum_{k,p=1}^nsR_{p0jq}\,\alpha_i(v_1,\ldots,e^q(v_k).e_{p+n},\ldots,v_n) \\
 &= \sum_{q=0}^n\sum_{k,p=1}^n\sum_{\sigma\in S_{n}:\ \sigma_1=k}\frac{1}{(n-1)!}
sR_{p0jq}(-1)^{k-1}\sg(\tilde{\sigma})e^q(v_{\sigma_1})\,\alpha_i(e_{p+n},v_{\sigma_2},\ldots,v_{\sigma_n}) \\
&= \sum_{q=0,\ p=1}^nsR_{p0jq}\,e^{q}\wedge
 e_{p+n}\lrcorner\alpha_i\ (v_1,\ldots,v_n)   
\end{split}
\end{equation*}
where $\tilde{\sigma}$ above is the permutation which transforms $\sigma_2,\ldots,\sigma_n$ back into the ordered set $1,\ldots,\widehat{k},\ldots,n$. Finally the tensors introduced in (\ref{dalphai},\ref{Ralphai}) are coherent with the
computation of $\dx\alpha_i$ from above. Indeed,
\begin{equation}\label{formulaRxialphai}
 \begin{split}
  \calR\alpha_i\ =&\  \sum_{j=0}^ne^j\wedge\frac{1}{2}
 \sum_{k=1}^n\alpha_i(\ldots,\calR(e_j,\ ),\ldots)  \\
=&\ \sum_{0\leq j<q\leq n}\sum_{p=1}^nsR_{p0jq}\,e^{jq}\wedge
e_{p+n}\lrcorner\alpha_i 
 \end{split}
\end{equation}
as wished.
\end{proof}
The previous formula may be partly simplified if (\ref{alphaicombes}) is used:
\begin{equation}\label{partlysimplified}
\begin{split}
\lefteqn{e_{p+n}\lrcorner\alpha_i\ = }\\
&= n_i \sum_{k=1}^n(-1)^{k-1}\sum_{\sigma\in S_n} 
e^{n+1}\circ B_{\sigma_1}\wedge\cdots\wedge e_{p+n}\lrcorner(e^{k+n}\circ
B_{\sigma_k})\wedge \cdots\wedge e^{2n}\circ B_{\sigma_n} \\
&= n_i\sum_{k=1}^n\sum_{\sigma_k\geq n-i+1}(-1)^{k-1}\,
e^{n+1}\circ B_{\sigma_1}\wedge\cdots\wedge \delta_{pk}\wedge\cdots
\wedge e^{2n}\circ B_{\sigma_n} \\
&= n_i\sum_{\sigma\in S_n:\: \sigma_p\geq n-i+1}(-1)^{p-1}\, e^{n+1}\circ
B_{\sigma_1}\wedge\cdots\wedge \widehat{e^{n+p}\circ
B_{\sigma_p}}\wedge\cdots\wedge e^{2n}\circ B_{\sigma_n} .
\end{split}
\end{equation}

The cases $\calR\alpha_0$ and $\calR\alpha_{1}$ having the particular
expressions appearing in (\ref{dalphazero},\ref{dalphaum}) may now be deduced.
Clearly, by \eqref{formulaRxialphai},
\begin{equation}
 \calR\alpha_0=0.
\end{equation}
Recurring to the formula above, we find
\begin{eqnarray*}
\lefteqn{\calR\alpha_{1}\ =}\\
&=& \sum_{0\leq j<q\leq n}\sum_{p=1}^nsR_{p0jq}\,e^{jq}\wedge
e_{p+n}\lrcorner\alpha_{1} \\
&=& n_1\sum_{0\leq j<q\leq n}\sum_{p=1}^nsR_{p0jq}\,e^{jq}\wedge  
\sum_{k=1}^n\sum_{\sigma\in S_n:\:\sigma_k=n}(-1)^{k-1}\delta_{pk}\,
e^{1}\wedge\cdots\widehat{e^k}\cdots\wedge e^{n} \\
&=& \frac{1}{(n-1)!}\sum_{0\leq j<q\leq n}\sum_{k=1}^nsR_{k0jq}\,e^{jq}\wedge
(-1)^{k-1}(n-1)!\,e^{1}\wedge\cdots\widehat{e^k}\cdots\wedge e^{n} \\
&=& \sum_{q=1}^nsR_{q00q}\,e^{0}\wedge e^{1}\wedge\cdots\wedge{e^q}\wedge\cdots\wedge
e^{n} \\
&=& -\frac{1}{s}\ric(\xi,\xi)\,\vol .
\end{eqnarray*}

\subsection{Proofs for Section \ref{subsection_Somefunctorialproperties}}

We just have to finish the proof of Theorem \ref{teorema_restrictiontoTGeodhypers}. Let us prove the first equality in \eqref{restrictiontoTGeodhypers_formula}. Clearly, using the mirror map, $B\vec{n}^v=0$. On the other hand, $\vec{n}^h\lrcorner\alpha_n=0$. Now we consider again formula (\ref{alphaicombes}) with the same $B_j$:
\[ \vec{n}^v\lrcorner\alpha_i=n_i\,\vec{n}^v\lrcorner\sum_{\sigma\in S_n}e^{n+1}\circ B_{\sigma_1}\wedge\cdots\wedge e^{2n}\circ B_{\sigma_n} . \]
We may certainly assume $\vec{n}^h=e_1$, so that $e^{n+j}(\vec{n}^v)=\delta_{j1}$, and proceed
\begin{eqnarray*}
 &=& n_i\sum_{k=n-i+1}^n\sum_{\sigma\in S_n:\:\sigma_1=k}
 e^{n+2}\circ B_{\sigma_2}\wedge\cdots\wedge e^{2n}\circ B_{\sigma_n} \\
&=&-n_ii\,\alpha_{n-1}^N\circ(B^{n-i}\wedge1^{i-1}) \\
&=& -\frac{1}{(n-i)!(i-1)!}\,\alpha_{n-1}^N\circ(B^{n-1-(i-1)}\wedge1^{i-1})\\
&=& -\alpha_{i-1}^N  .
 \end{eqnarray*}
For the second equality in \eqref{restrictiontoTGeodhypers_formula}, which is very similar and straightforward to deduce as the first, we return to canonical methods. Let $v_1,\ldots,v_{n-1}$ be any vector fields on $S_sN$. Then we define $w_1=\vec{n}^h,w_2=v_1,\ldots,w_n=v_{n-1}$ and notice $Bw_1=\vec{n}^v$. As such,
\begin{align*}
  &\ \ \ \vec{n}^h\lrcorner\alpha_{i}(v_1,\ldots,v_{n-1}) \\
 & = \alpha_i(w_1,w_2,\ldots,w_n) \\
 & = n_i\sum_{j=1}^{n-i}\sum_{\sigma:\:\sigma_{j}=1}\sg(\sigma)\, \alpha_n (Bw_{\sigma_1},\ldots,Bw_{\sigma_j},\ldots,Bw_{\sigma_{n-i}},w_{\sigma_{n-i+1}}, \ldots,w_{\sigma_n})\\
 & = n_i\sum_{j=1}^{n-i}\sum_{\sigma:\:\sigma_{j}=1}\sg(\sigma)(-1)^{j-1}\, \alpha_n (\vec{n}^v,Bw_{\sigma_1},\ldots,\widehat{Bw_{\sigma_{j}}},\ldots,     Bw_{\sigma_{n-i}},w_{\sigma_{n-i+1}}, \ldots,w_{\sigma_n}) \\
 & = n_i\sum_{j=1}^{n-i}\sum_{\tau\in S_{n-1}}\sg(\tau)\,(\vec{n}^v\lrcorner
\alpha_n)(Bw_{\tau_2},\ldots,\ldots,Bw_{\tau_{n-i}},w_{\tau_{n-i+1}}, \ldots,w_{\tau_n}) \\
 & = -\frac{n-i}{i!(n-i)(n-i-1)!}\sum_{\tau}\sg(\tau)\,\alpha_{n-1}^N(Bw_{\tau_2},\ldots,\ldots,Bw_{\tau_{n-i}},w_{\tau_{n-i+1}}, \ldots,w_{\tau_n} )  \\
& = -\alpha_{i}^N(w_2,\ldots,w_n) 
\end{align*}
and the result follows.

Now let us prove Theorem \ref{metricaEinsteinealpha2}.
\begin{proof}[Proof of formula \eqref{dasteriscoalphadois}]
 Clearly from \eqref{Basicstrutequations2}, \eqref{dalphai} and \eqref{Ralphai}
we have
\begin{eqnarray*}
\lefteqn{ \dx*\alpha_{n-2} \: =\:-\frac{1}{s}\,\theta\wedge\dx\alpha_{2} }\\
&=&-\frac{1}{s}\, \theta\wedge\calR\alpha_{2} \\
&=& -\,\theta\wedge\sum_{1\leq j<q\leq n}
\sum_{p=1}^nR_{p0jq}\,e^{jq}\wedge e_{p+n}\lrcorner\alpha_{2} \\
&=& -n_2\,\theta\wedge\sum_{1\leq j<q\leq n} \sum_{p,a,b=1}^nR_{p0jq}\,e^{jq}\sum_{\sigma\in S'_n} e_{p+n}\lrcorner(e^{n+1}\circ
B_{\sigma_1}\wedge\cdots\wedge e^{2n}\circ B_{\sigma_n})
\end{eqnarray*}
As usual, here we let $B_1=\cdots=B_{n-2}=B,\ B_{n-1}=B_n=1$, and then continue the
computation (we let $S'_n=\{\sigma\in S_n:\ \sigma_a=n-1,\:\sigma_b=n\}$):
\begin{equation*}
 \begin{split}
&= -\frac{1}{2}\,\theta\wedge\sum_{1\leq j<q\leq n}\sum_{p=1}^nR_{p0jq}\,e^{jq}
\biggl(\sum_{a<b}e_{p+n}\lrcorner (e^{1}\wedge\cdots\wedge
e^{n+a}\wedge\cdots\wedge e^{n+b}\wedge\cdots\wedge e^{n}) \biggr. \\
& \hspace{20mm} \biggl.+\sum_{b<a} e_{p+n}\lrcorner (e^{1}\wedge\cdots\wedge
e^{n+b}\wedge\cdots\wedge e^{n+a}\wedge\cdots\wedge e^{n}) \biggr)  \ \ \ \ \ \mbox{(cont.)}
\end{split}
\end{equation*}
\begin{equation*}
 \begin{split}
&= -\,\theta\wedge\sum_{1\leq j<q\leq n}\sum_{a<b}\sum_{p=1}^nR_{p0jq}\,e^{jq}
e_{p+n}\lrcorner (e^{1}\wedge\cdots\wedge
e^{n+a}\wedge\cdots\wedge e^{n+b}\wedge\cdots\wedge e^{n})  \\
&= \theta\wedge\sum_{1\leq j<q\leq n}\sum_{a<b}\biggl( R_{a0jq}\,
e^{jq}(-1)^{a+b}e^{n+b}\wedge e^{1}\wedge\cdots\wedge\widehat{e^{a}}\wedge\cdots\wedge
\widehat{e^{b}}\wedge\cdots\wedge e^{n} \biggr. \\
&  \hspace{20mm} \biggl.-R_{b0jq}\,e^{jq}(-1)^{a+b}e^{n+a}\wedge
e^{1}\wedge\cdots\wedge\widehat{e^{a}}\wedge\cdots\wedge
\widehat{e^{b}}\wedge\cdots\wedge e^{n} \biggr) \\
&= s\,e^0\wedge\sum_{a<b}(-R_{a0ab}\,e^{n+b}+R_{b0ab}\,e^{n+a})\wedge e^{1\cdots n}\\
&=(\sum_{a<b}+\sum_{b<a})s\,R_{a0ab}\,e^{n+b}\wedge\vol\\
&=\sum_{a,b=1}^ns\,R_{a0ab}\,e^{n+b}\wedge\vol .
\end{split}
\end{equation*}
Recalling \eqref{defofrho}, we may conclude.
\end{proof}

\subsection{Proofs for Section \ref{Atte}}

\begin{proof}[Proof of formula \eqref{dalphaicsc} in Section \ref{Atte}]
For constant sectional curvature, which we recall is equi\-valent to
$R_{qpij}=c(\delta_{iq}\delta_{jp}-\delta_{ip}\delta_{jq})$, one immediately finds from
formula \eqref{formulaRxialphai} above that
\begin{equation*}
 \calR\alpha_i = -c\,\theta\wedge\sum_{q=1}^ne^q\wedge e_{q+n}\lrcorner\alpha_i .
\end{equation*}
Now, assuming $B_1=\ldots=B_{n-i}=B$ and $B_{n-i+1}=\ldots=B_n=1$, the relevant component is
\begin{eqnarray*}
\lefteqn{ \sum_{q=1}^ne^q\wedge e_{q+n}\lrcorner\alpha_i \ =}\\
&=& n_i\sum_{q=1}^ne^q\wedge e_{q+n}\lrcorner \sum_{\sigma\in S_n}e^{n+1}\circ
B_{\sigma_1}\wedge
\cdots\wedge e^{2n}\circ B_{\sigma_n}\\
&=& n_i\sum_{q=1}^n\sum_\sigma e^{n+1}\circ B_{\sigma_1}\wedge
\cdots\wedge e^{n+q}\circ B_{\sigma_q}(e_{n+q})e^q\wedge\cdots\wedge e^{2n}\circ
B_{\sigma_n}\\
&=& n_i\sum_{q}\sum_{\sigma:\ \sigma_q>n-i} e^{n+1}\circ B_{\sigma_1}\wedge
\cdots\wedge e^q\wedge\cdots\wedge e^{2n}\circ B_{\sigma_n}\qquad
\mbox{(cf. \eqref{partlysimplified})} \\
&=& in_i\sum_{q}\sum_{\sigma:\ \sigma_q=n-i+1} e^{n+1}\circ B_{\sigma_1}\wedge
\cdots\wedge e^q\wedge\cdots\wedge e^{2n}\circ B_{\sigma_n} .
\end{eqnarray*}
Here we may change to $B_1=\ldots=B_{n-i}=B_{n-i+1}=B$ and $B_{n-i+2}=\ldots=B_n=1$
and then, resuming the computation,
\begin{equation*}
\begin{split}
&= in_i\sum_{q=1}^n\sum_{\tau\in S_n:\ \tau_q=n-i+1} e^{n+1}\circ
B_{\tau_1}\wedge \cdots\wedge e^{n+q}\circ B_{\tau_q}\wedge\cdots\wedge e^{2n}\circ
B_{\tau_n}\\
&= \frac{in_i}{n_{i-1}}\,\alpha_{i-1} \\
&= (n-i+1)\,\alpha_{i-1} .
\end{split}
\end{equation*}
Formula 
$\dx\alpha_i=\frac{i+1}{s^2}\,\theta\wedge\alpha_{i+1}+\calR\alpha_i=\theta\wedge
(\frac{i+1}{s^2}\,\alpha_{i+1}-c(n-i+1)\,\alpha_{i-1})$ follows.
\end{proof}

\subsection{Proofs for Section \ref{ELsottsb}}
\label{PfsELsystems}

In Section \ref{ELsottsb} the reader finds a statement about the pullback of the Lagrangians $\alpha_i$ by the lift $\widehat{f}$ to $\tsb$ of an isometric immersion $f:N\rr M$ and a subsequent formula
$\widehat{f}^*\alpha_i=(-1)^{i}\sympol_{i}(A)\,\vol_N$ which is worth checking in
detail.
\begin{proof}[Proof of Proposition \ref{pullbackhatalphai}]
We have seen that
\begin{eqnarray*}
 \alpha_{n-i} &=&
n_{i}\sum_{\sigma\in S_n}\mathrm{sg}(\sigma)\,e^{\sigma_1}\wedge\cdots\wedge
e^{\sigma_{i}}\wedge e^{(n+\sigma_{i+1})}\wedge\cdots\wedge e^{(n+\sigma_n)} .
\end{eqnarray*}
Hence, using \eqref{pulbacksdeadaptedcoframeum} and \eqref{pulbacksdeadaptedcoframedois} and permutations $\tau$ on the set
$\{i+1,\ldots, n\}$,
\begin{equation*}
 \begin{split}
\lefteqn{(-1)^{n-i}\,\widehat{f}^*\alpha_{n-i}\ = }\\
&= n_{i}\sum_{\sigma\in S_n}\sum_{j=1}^{n-i}\sum_{k_j=1}^{\ n}
\mathrm{sg}(\sigma)\,A^{\sigma_{i+1}}_{k_1}\cdots
A^{\sigma_n}_{k_{n-i}}\,e^{\sigma_1}\wedge\cdots\wedge e^{\sigma_{i}}\wedge
e^{k_{1}}\wedge\cdots\wedge e^{k_{n-i}}\\
 &= n_{i}\sum_{\sigma}\sum_{\tau\in
S_{n-i}}\mathrm{sg}(\sigma)\,A^{\sigma_{i+1}}_{\sigma_{\tau_{i+1}}}\cdots
A^{\sigma_n}_{\sigma_{\tau_{n}}}\,e^{\sigma_1}\wedge\cdots\wedge e^{\sigma_{i}}\wedge
e^{\sigma_{\tau_{i+1}}}\wedge\cdots\wedge e^{\sigma_{\tau_{n}}}\\
&= n_{i}\sum_{\sigma}\sum_{\tau\in
S_{n-i}}\mathrm{sg}(\tau)\,A^{\sigma_{i+1}}_{\sigma_{\tau_{i+1}}}\cdots
A^{\sigma_{n}}_{\sigma_{\tau_{n}}}\,e^{1}\wedge\cdots\wedge e^{n} .
 \end{split}
\end{equation*}
Letting $e_1,\ldots,e_n$ be a direct orthonormal basis of
eigenvectors of $A$ with eigenvalues $\lambda_j$ (recall $A$ is symmetric), we see
\begin{eqnarray*}
\lefteqn{ (-1)^{n-i}\,\widehat{f}^*\alpha_{n-i}\ =\ n_i\sum_{\sigma\in
S_n}\lambda_{\sigma_1}\cdots\lambda_{\sigma_{n-i}}\,\vol_N\ =}\\ 
&=& \sum_{1\leq j_1<\cdots<j_{n-i}\leq n}
\lambda_{j_1}\cdots\lambda_{j_{n-i}}\,\vol_N\ =\ \sympol_{n-i}(A)\,\vol_N
\end{eqnarray*}
as we wished.
\end{proof}

%\medskip

%\vspace*{34mm}

%\bibliographystyle{amsalpha}
%\bibliography{AlbuquerquesBibliography}

\begin{thebibliography}{30}




%\bibitem{Alb2}
%R.~Albuquerque,
%\emph{On the {$\mathrm{G}_2$} bundle of a {R}iemannian 4-manifold},
%J. Geom. Phys. \textbf{60} (2010), 924--939.

\bibitem{Alb5}
R.~Albuquerque,
\emph{Curvatures of weighted metrics on tangent sphere bundles},
Riv. Mat. Univ. Parma, {2} (2011), no.~2, 299--313.

\bibitem{Alb3}
---,
\emph{Weighted metrics on tangent sphere bundles},
Quart. J. Math., {63} (2012), no.~2, 259--273.

\bibitem{Alb2013b}
---,
\emph{Variations of gwistor space},
Portugaliae Math. {70} (2) (2013), 145--160.

\bibitem{Alb4}
---,
\emph{Homotheties and topology of tangent sphere bundles},
Jour. of Geometry, {105} (2014), no.~2, 327--342.


\bibitem{Alb2015a}
---,
\emph{A fundamental differential system of 3-dimensional Riemannian geometry},
Bull. Sci. Math. 143 (2018) 82--107.

\bibitem{AlbSal1}
R.~Albuquerque and I.~Salavessa,
\emph{The {$\mathrm{G}_2$} sphere of a 4-manifold},
Monatsh. Mathematik, {158} (2009), no.~4, 335--348.

%\bibitem{AlbSal2}
%---, \emph{Erratum to: The {$\mathrm{G}_2$} sphere of a 4-manifold},
%Monatsh. Mathematik, {160} (2010), no.~1, 109--110.


\bibitem{BCGGG}
R.~Bryant, S.~S.~Chern, R.~Gardner, H.~Goldschmidt, and Ph.~Griffiths,
\emph{Exterior differential systems},
vol.~18, MSRI Publications, New York, 1991.

\bibitem{BGG}
R.~Bryant, Ph.~Griffiths, and D.~Grossman,
\emph{Exterior differential systems and {E}uler­-{L}agrange partial differential equations},
University of Chicago Press, 2003.

\bibitem{Blair}
D.~Blair,
\emph{Riemannian geometry of contact and symplectic manifolds},
Progress in Mathematics, vol. 203, Birkha\"user Boston Inc., Boston, MA, 2002.

\bibitem{Gri1}
Ph.~Griffiths,
\emph{Exterior differential systems and the calculus of variations},
Progress in Mathematics, vol.~25, Birkha\"user, Boston, Basel, Stuttgart, 1983.

\bibitem{Gri2}
---,
\emph{{S}elected {W}orks of {P}hillip {A}. {G}riffiths with {C}ommentary},
vol. Part 4 ``Differential Systems'', American Mathematical Society, Providence, R.I., 2003.


\bibitem{Hsiung}
C.-C. Hsiung,
\emph{Some integral formulas for closed hypersurfaces in Riemannian space},
Pacific J. Math. 6 (1956), 291--299.


\bibitem{IveyLan}
Th.~Ivey and J.~Landsberg,
\emph{Cartan for beginners: differential geometry via moving frames and exterior differential systems},
Graduate Studies in Mathematics, vol.~61, American Mathematical Society, Providence, RI, 2003.

\bibitem{Joy}
D.~Joyce,
\emph{Riemannian {H}olonomy {G}roups and {C}alibrated {G}eometry},
Oxford Graduate Texts in Mathematics, Oxford University Press, 2009.

\bibitem{Katsurada}
 Y. Katsurada,
\emph{Generalized Minkowski formulas for closed hypersurfaces in Riemann space},
Ann. Mat. Pura Appl. (4) 57 (1962), 283--293.

\bibitem{Rei}
R.~C.~Reilly,
\emph{Variational properties of functions of the mean curvatures for hypersurfaces in space forms},
J. Differential Geometry 8 (1973), 465--477.


\bibitem{Sakai}
T.~Sakai,
\emph{Riemannian {G}eometry},
Transl. Math. Mono., vol. 149, AMS, 1996.

\bibitem{SingerThorpe}
I.~Singer and J.~A.~Thorpe,
\emph{Lecture Notes on Elementary Topology and Geometry},
Springer-Verlag, 1967.

\bibitem{Tash}
Y.~Tashiro,
\emph{On contact structures on tangent sphere bundles},
T\^ohoku Math. Jour. 21 (1969), 117--143.

\bibitem{Via}
J.~A.~Viaclovsky,
\emph{Conformal geometry, contact geometry, and the calculus of variations},
Duke Math. J. 101, no. 2 (2000), 283--316.


\bibitem{Xin}
Y.~Xin,
\emph{Minimal submanifolds and related topics},
Nankai Tracts in Mathematics, vol.~8, World Scientific, 2003.

\end{thebibliography}

\vspace{6mm}

\ \\
The author acknowledges the support of Funda\c{c}\~{a}o Ci\^{e}ncia e Tecnologi\-a,
Portu\-gal, through CIMA-U\'E and the Project PTDC/MAT/118682/2010.

\ \\
R. Albuquerque\hspace{9mm} $|$\hspace{9mm}  \texttt{rpa@uevora.pt}\\
Departamento de Matem\'atica da Universidade de \'Evora
e Centro de Investiga\c c\~ao em Mate\-m\'a\-ti\-ca e Aplica\c c\~oes (CIMA-U\'E), Rua Rom\~ao Ramalho, 59, 671-7000 \'Evora, Portugal.

%\end{color}

\end{document}